\documentclass[11pt]{amsart}
\usepackage[arrow,matrix,graph,frame,poly,arc,tips]{xy}
\usepackage{amscd,amssymb,subfigure,epsfig,color}
\usepackage[colorlinks,plainpages,backref,urlcolor=blue]{hyperref}
\usepackage[width=5.65in,height=7.9in,centering]{geometry}
\usepackage{amsmath}
\usepackage{fancyvrb,bbm,subfigure}
\usepackage{tikz}
\usepackage{xcolor}

\newtheorem{theorem}{Theorem}[section]
\newtheorem{lemma}[theorem]{Lemma}
\newtheorem{corollary}[theorem]{Corollary}
\newtheorem{proposition}[theorem]{Proposition}

\theoremstyle{definition}
\newtheorem{definition}[theorem]{Definition}
\newtheorem{examplex}[theorem]{Example} % let's adjust these environments
\newtheorem{remarkx}[theorem]{Remark}   % to end with a qed symbol
\newtheorem{question}[theorem]{Question}

\newtheorem*{ack}{Acknowledgment}

\newenvironment{remark}%
{\pushQED{\qed}\begin{remarkx}}%
{\popQED\end{remarkx}}

\newenvironment{example}%
{\pushQED{\qed}\begin{examplex}}%
{\popQED\end{examplex}}

\newcommand{\abs}[1]{{\lvert #1 \rvert}}

\newcommand{\A}{{\mathcal A}}
\newcommand{\B}{{\mathcal B}}

\newcommand{\XX}{{\mathcal X}}  % multiple points in a multinet

\newcommand{\Pl}{{\mathcal P}}  %% a pencil of lines

\newcommand{\N}{{\mathbb N}}
\newcommand{\Z}{{\mathbb Z}}
\newcommand{\Q}{{\mathbb Q}}

\newcommand{\C}{{\mathbb C}}

\newcommand{\F}{{\mathbb F}}
\renewcommand{\k}{\Bbbk}
\newcommand{\K}{{\mathbb K}}

\renewcommand{\P}{{\mathbb P}}

\newcommand{\bo}{{\mathbf 1}}

\newcommand{\G}{\widehat{G}}  % character group
\newcommand{\V}{{\mathcal V}}  % characteristic varieties
\newcommand{\Mat}{{\mathcal M}} % matroid category
\newcommand{\FF}{{\mathcal F}} % fg free abelian groups category

\newcommand{\PMat}{{\mathcal M^+}} % pointed matroid category
\newcommand{\unit}{{\mathbbm 1}} % matroid unit
\newcommand{\id}{{\mathrm{id}}} % identity map
\newcommand{\Ab}{{h}}     % free abelian group on points of matroid
\newcommand{\PAb}{{\overline{h}}}     % projective version

\DeclareMathOperator{\rank}{rank}

\DeclareMathOperator{\Hom}{Hom}

\DeclareMathOperator{\Gal}{Gal}

\DeclareMathOperator{\pc}{\|}    % parallel connection
\DeclareMathOperator{\ch}{char}  % characteristic of a field
\DeclareMathOperator{\poin}{Poin} % Poincare poly; change to taste
\DeclareMathOperator{\plugin}{\circ} % composition operation
\DeclareMathOperator{\orb}{orb}
\DeclareMathOperator{\ab}{ab}

\DeclareMathOperator{\im}{im}
\DeclareMathOperator{\pr}{pr}

\newcommand{\surj}{\twoheadrightarrow}
\newcommand{\inj}{\hookrightarrow}

\newcommand{\co}{\!:\!}

\def\set#1{\{#1\}}

\def\PC#1#2{{\sideset{_{\!#1\!}}{_{\!#2}}\pc}}  % with indices

\begin{document}

\title[{M}ilnor fibrations of arrangements]%
{Multinets, parallel connections, and {M}ilnor fibrations 
of arrangements}

\author[G. Denham]{Graham Denham$^1$} 
\address{Department of Mathematics, University of Western Ontario,
London, ON  N6A 5B7}
\email{\href{mailto:gdenham@uwo.ca}{gdenham@uwo.ca}}
\urladdr{\href{http://www.math.uwo.ca/~gdenham}%
{http://www.math.uwo.ca/\~{}gdenham}}

\author[A. Suciu]{Alexander~I.~Suciu$^2$}
\address{Department of Mathematics,
Northeastern University,
Boston, MA 02115}
\email{\href{mailto:a.suciu@neu.edu}{a.suciu@neu.edu}}
\urladdr{\href{http://www.math.neu.edu/~suciu/}%
{http://www.math.neu.edu/\~{}suciu}}

\thanks{$^1$Supported by NSERC}
\thanks{$^2$Partially supported by NSF grant DMS--1010298}


\begin{abstract} 
The characteristic varieties of a space are the jump loci for 
homology of rank $1$ local systems.  The way in which the 
geometry of these varieties may vary with the characteristic 
of the ground field is reflected in the homology of finite cyclic covers.  
We exploit this phenomenon to detect torsion in the homology 
of Milnor fibers of projective hypersurfaces. 
One tool we use is the interpretation of the degree $1$ characteristic 
varieties of a hyperplane arrangement complement in terms of 
orbifold fibrations and multinets on the corresponding matroid.  
Another tool is a polarization construction, based on the parallel 
connection operad for matroids. Our main result gives a 
combinatorial machine for producing arrangements whose 
Milnor fibers have torsion in homology.  In particular, 
this shows that Milnor fibers of hyperplane arrangements do 
not necessarily have a minimal cell structure.  
\end{abstract}


\maketitle
\setcounter{tocdepth}{1}
\tableofcontents

\section{Introduction}
\label{sect:intro}

\subsection{The Milnor fibration}
\label{intro:mf}

A classical construction, due to Milnor \cite{Mi}, associates 
to every homogeneous polynomial $f\in \C[z_1,\dots , z_\ell]$ 
a fiber bundle, with base space $\C^*=\C\setminus \set{0}$, 
total space the complement in $\C^{\ell}$ to the hypersurface 
defined by $f$, and projection map 
$f\colon \C^{\ell} \setminus f^{-1}(0)\to \C^*$.

The Milnor fiber $F=f^{-1}(1)$ has the homotopy type of a finite,
$(\ell-1)$-dimensional CW-complex, while the monodromy of the 
fibration, $h\colon F\to F$, is given by $h(z)=e^{2\pi i/n} z$.
If $f$ has an isolated singularity at the origin, then $F$ is 
homotopic to a bouquet of $(\ell-1)$-spheres, 
whose number can be determined by algebraic means. 
In general, though, it is a rather hard problem to compute 
the homology groups of the Milnor fiber, even in the case when 
$f$ completely factors into distinct linear forms, so that the 
hypersurface $\set{f=0}$ is a hyperplane arrangement.

Building on our previous work with D.~Cohen \cite{CDS03}, 
we show there exist projective hypersurfaces (indeed, hyperplane 
arrangements) whose complements have torsion-free 
homology, but whose Milnor fibers have torsion in homology. 
Our main result can be summarized as follows. 

\begin{theorem}
\label{thm:intro1}
For every prime $p\ge 2$, there is a hyperplane arrangement 
whose Milnor fiber has non-trivial $p$-torsion in homology. 
\end{theorem}
  
This resolves a problem posed by Randell~\cite[Problem 7]{Ra11},
who conjectured that Milnor fibers of hyperplane arrangements 
have torsion-free homology.  Our examples also give a 
refined answer to the question posed by Dimca and 
N\'emethi~\cite[Question~3.10]{DN04}: torsion is possible 
even if the hypersurface is defined by a reduced equation.

Additionally, our result shows that Milnor fibers of hyperplane 
arrangements do not necessarily have a minimal cell structure.  
Again, this stands in marked contrast with arrangement complements, 
which all admit perfect Morse functions.

Our method also allows us to compute the homomorphism induced 
in homology by the monodromy of the Milnor fibration, with coefficients 
in a field of characteristic not dividing the order of the monodromy.  
For some related recent work in this direction, we refer to 
\cite{BDS11, Li12, Wi2}.

A much-studied question in the subject is whether the Betti 
numbers of the Milnor fiber of an arrangement  $\A$ are determined 
by the intersection lattice, $L(\A)$. While we do not address this 
question directly here, our result raises a related, and arguably 
even subtler question.

\begin{question}
\label{q:comb tors}
Is the torsion in the homology of the Milnor fiber of a hyperplane 
arrangement combinatorially determined? 
\end{question}

As a preliminary question, one may also ask: can one predict 
the existence of torsion in the homology of the Milnor fiber $F(\A)$ 
simply by looking at $L(\A)$?  As it turns out, under fairly general 
assumptions, the answer is yes: if $L(\A)$ satisfies certain very 
precise conditions, then automatically $H_*(F(\A),\Z)$ will 
have non-zero torsion, in a combinatorially determined degree.

\subsection{Outline of the proof}
\label{intro:outline}
Let $\A$ be a (central) arrangement of $n$ hyperplanes in $\C^{\ell}$, 
defined by a polynomial $Q(\A)=\prod_{H\in \A} f_H$, where each 
$f_H$ is a linear form whose kernel is $H$.   The starting point 
of our study is the well-known observation that the Milnor fiber 
of the arrangement, $F(\A)$, is a cyclic, $n$-fold regular cover 
of the projectivized complement, $U(\A)$; this cover is 
defined by the homomorphism $\delta\colon \pi_1(U(\A))\surj \Z_n$, 
taking each meridian generator $x_H$ to $1$. 

Now, if $\k$ is an algebraically closed field whose characteristic 
does not divide $n$, then $H_q(F(\A),\k)$ decomposes as a 
direct sum, $\bigoplus_{\rho} H_q(U(\A),\k_{\rho})$, where the 
rank $1$ local systems $\k_{\rho}$ are indexed by characters 
$\rho\colon \pi_1(U(\A))\to \k^*$ that factor through $\delta$.  
Thus, if there is such a character $\rho$ for which  
$H_q(U(\A),\k_{\rho})\ne 0$, where $\k$ is a field 
of characteristic $p>0$ not dividing $n$, but there is no 
corresponding character in characteristic $0$, then the 
group $H_q(F(\A),\Z)$ will have non-trivial $p$-torsion.

To find such characters, we proceed in several steps.  First, 
we cast a wider net, and look at multi-arrangements, that is, 
pairs $(\A,m)$, with positive integer weights $m_H$ attached 
to each hyperplane $H\in \A$.  The corresponding (unreduced) Milnor 
fiber, $F(\A,m)$, is the cover defined by the homomorphism 
$\delta_m\colon \pi_1(U(\A))\surj \Z_N$, $x_H\mapsto m_H$, 
where $N$ denotes the sum of the weights. Fix a prime $p$. 
Starting with an arrangement $\A$ supporting a suitable multinet 
structure, we find a deletion $\A' =\A\setminus \{H\}$, and  
a choice of multiplicities $m'$ on $\A'$ such that $H_1(F(\A',m'),\Z)$ 
has $p$-torsion. Finally, we construct a ``polarized" arrangement, 
$\B=\A' \pc m'$, and show that $H_*(F(\B),\Z)$ has $p$-torsion.

\subsection{Characteristic varieties}
\label{intro:cv}

Let us now provide a bit more detail on the structure of the 
paper, and the ingredients that go into the proof of the 
main theorem.
We start in Sections \ref{sect:cv}  and \ref{sect:tors covers} with 
a general study of the jump loci for homology in rank $1$ local 
systems.  The geometry of these varieties, and the way it changes 
with the characteristic of the ground field affects the homology 
of finite cyclic covers. 

Let $X$ be a connected, finite-type CW-complex, with fundamental 
group $G=\pi_1(X,x_0)$.  The {\em characteristic varieties}\/ of 
$X$ are the subvarieties $\V^q_d(X,\k)$ of the character group 
$\widehat{G}=\Hom(G,\k^*)$, consisting of those characters 
$\rho$ for which $H_q(X,\k_{\rho})$ has dimension at least $d$.  
(For simplicity, we drop the subscript if $d=1$.)  

As is well-known, the characteristic varieties control to a large extent 
the homology of finite abelian covers.  More precisely, if $X^{\chi}\to X$ 
is a regular cover defined by an epimorphism $\chi\colon G\to A$, 
where $A$ is a finite abelian group, and if $\k$ is an algebraically closed 
field of characteristic $p$, where $p\nmid \abs{A}$, then 
$H_q(X^{\chi},\k)$ has dimension equal to 
$\sum_{d\geq 1}\abs{\im(\widehat{\chi}_\k) \cap  \V_d^q(X,\k)}$, 
where $\widehat{\chi}_\k\colon \widehat{A} \to \widehat{G}$ 
is the induced morphism between character groups.

\begin{theorem}
\label{thm:intro ptors1}
Let $X^{\chi}\to X$ be a regular, finite cyclic cover, defined 
by an epimorphism $\chi\colon \pi_1(X)\surj \Z_r$. 
Suppose that $\im(\widehat{\chi}_\C)\not\subseteq \V^q(X,\C)$, but
$\im(\widehat{\chi}_\k) \subseteq\V^q(X,\k)$, 
for some field $\k$ of characteristic $p$ not dividing $r$. Then
$H_q(X^{\chi},\Z)$ has non-zero $p$-torsion.  
\end{theorem}

Furthermore, if $h\colon X^{\chi} \to X^{\chi}$ is the monodromy 
of the cover,  then the characteristic polynomial of the algebraic 
monodromy acting on $H_q(X^{\chi},\k)$ can be written down 
explicitly in terms of the intersection inside $\widehat{G}$ 
of the varieties $\V_d^q(X,\k)$ with the cyclic subgroup 
$\im(\widehat{\chi}_\k)$.

In order to apply Theorem \ref{thm:intro ptors1} in full generality, one 
needs to know both the characteristic varieties over $\C$ and 
over $\k$, and exploit the qualitative differences between the two. 
In a special situation (which will occur often in our approach), 
we only need to verify that a certain condition on the first characteristic 
variety over $\C$ holds, in which case we can construct infinitely 
such covers, starting from a suitable character.

\begin{theorem}
\label{thm:intro ptors2}
Suppose there is a character which factors as 
$\pi_1(X)\surj \Z * \Z_p \surj \Z \to \C^*$, for some prime $p$, 
but does not belong to $\V^1(X,\C)$.  
Then, for all sufficiently large integers $r$ not divisible by $p$, 
there exists a regular, $r$-fold cyclic cover $Y\to X$ such that 
$H_1(Y,\Z)$ has non-zero $p$-torsion.
\end{theorem}

In the above, $\Z* \Z_p=\langle a, b \mid b^p=1\rangle$  is 
the free product of the two factors, while $\Z * \Z_p \surj \Z$  
is the epimorphism given on generators by $a\mapsto 1$ 
and $b\mapsto 0$.

\subsection{Quasi-projective varieties and orbifold fibrations}
\label{intro:qp}
 
Now suppose $X$ is a smooth, quasi-projective variety. 
Although almost nothing is known about the characteristic varieties 
of $X$ in positive characteristic, the varieties $\V^q(X,\C)$ 
are fairly well understood, at least qualitatively. For instance, 
work of Arapura \cite{Ar}, as further refined in \cite{Di07, ACM13} 
shows that $\V^1(X,\C)$ is a union of torsion-translated 
algebraic subtori inside $\Hom(\pi_1(X),\C^*)$.

More precisely, each positive-dimensional component of 
$\V^1(X,\C)$ is obtained by pullback along an orbifold fibration 
$f\colon X \to (\Sigma,\mu)$, where $\Sigma$ is a 
(possibly punctured) Riemann surface with integer 
weights $\mu_i\ge 2$ attached to a set of marked points 
of finite size $\abs{\mu}$.  We say that the orbifold 
$(\Sigma,\mu)$ is {\em small}\/  if 
either $\Sigma=S^1\times S^1$ and $\abs{\mu} \geq2$, 
or $\Sigma=\C^*$ and $\abs{\mu} \geq1$. 
Using Theorem \ref{thm:intro ptors2}, we prove 
in Section \ref{sect:orbi} the following result.

\begin{theorem}
\label{thm:intro pencil}
Let $X$ be a smooth, quasi-projective variety.  Suppose there 
is a small orbifold fibration $f\colon X\to (\Sigma,\mu)$ 
and a prime $p$ dividing each $\mu_i$.  
Then, for all sufficiently large integers $r$ not divisible by $p$, 
there exists a regular, $r$-fold cyclic cover $Y\to X$ 
such that $H_1(Y,\Z)$ has non-zero $p$-torsion.  
\end{theorem}

For instance, if $M$ is the Catanese--Ciliberto--Mendes Lopes 
surface from \cite{CCM}, then $M$ admits an elliptic pencil 
with two multiple fibers, each of multiplicity $2$.  Thus, for each odd integer 
$r>1$, the surface $M$ admits a cyclic $r$-fold cover with $2$-torsion 
in first homology.

\subsection{Multinets}
\label{intro:multinets}

In Sections \ref{sect:arr} and  \ref{sect:cyclic arr}, we consider 
in detail the case when $X=M(\A)$ is the complement of a 
hyperplane arrangement.  As shown by Falk and Yuzvinsky  
in \cite{FY}, large orbifold fibrations supported by $M(\A)$ 
correspond to {\em multinets}\/ on $L(\A)$.  Each such multinet  
consists of a partition of $\A$ into at least $3$ subsets 
$\A_1,\ldots,\A_k$, together with an assignment of 
multiplicities, $m\colon \A\to \N$, and a set of rank 
$2$ flats, called the base locus, such that any two 
hyperplanes from different parts of the partition intersect 
in the base locus, and several technical conditions are satisfied:   
for instance, the sum of the multiplicities over each part $\A_i$  
is constant, and for each flat $Z$ in the base locus, the sum 
$n_{Z}:=\sum_{H\in\A_i\colon H\leq Z} m_H$ 
is independent of $i$.

For our purposes, it will be convenient to isolate a special 
class of multinets:  a {\em pointed multinet}\/  on $\A$ 
is a multinet structure, together with a distinguished hyperplane 
$H$ for which $m_{H}>1$, and $m_{H} \mid n_Z$ for each 
flat $Z\ge H$ in the base locus.  Given such a pointed multinet, 
we show that the complement of the deletion $\A':=\A\setminus \{H\}$ 
supports a small pencil. Consequently, for any prime $p$ dividing 
$m_{H}$, and any sufficiently large integer $r$ not divisible by 
$p$, there exists a regular, $r$-fold cyclic cover $Y\to U(\A')$ 
such that $H_1(Y,\Z)$ has $p$-torsion.

On the other hand, we also show that any finite cyclic cover 
of an arrangement complement is dominated by a Milnor 
fiber corresponding to a suitable choice of multiplicities. 
Putting things together, we obtain the following result.

\begin{theorem}
\label{thm:intro multi tors}
Suppose $\A$ admits a pointed multinet, with distinguished 
hyperplane $H$ and multiplicity vector $m$.  Let $p$ be a prime 
dividing $m_H$. There is then a choice of multiplicity vector $m'$ 
on the deletion $\A' =\A\setminus \{H\}$ such that 
$H_1(F(\A',m'),\Z)$ has non-zero $p$-torsion.
\end{theorem}

For instance, if $\A$ is the reflection arrangement of type 
$\operatorname{B}_3$, defined by the polynomial 
$Q=xyz(x^2-y^2)(x^2-z^2)(y^2-z^2)$, then $\A$ 
satisfies the conditions of the theorem, for 
$m=(2,2,2,1,1,1,1,1,1)$ and $H= \{z=0\}$.  
Choosing then multiplicities $m'=(2,1,3,3,2,2,1,1)$ 
on $\A'$ shows that $H_1(F(\A',m'),\Z)$ has non-zero 
$2$-torsion.  

More generally, the reflection arrangement of 
the full monomial complex reflection group, $\A(p,1,3)$,
admits a pointed multinet, for each prime $p$. 
Thus, we can detect $p$-torsion in the first homology 
of the Milnor fiber of a suitable multiarrangement on 
the deletion.   These results recover computations from \cite{CDS03}
and recast them in a more general framework.

\subsection{Parallel connections and polarizations}
\label{intro:parallel}

In Sections \ref{sect:operad} and \ref{sect:polar}, we develop 
the theory of parallel connections of matroids and arrangements
in a more general, operadic setting.  We note that the parallel
connection construction can be interpreted as a pseudo-operad
composition map (in the sense of \cite{MSS02}): Figure
\ref{fig:pc_example} may give the reader some intuition.  As shown by
Falk and Proudfoot \cite{FP}, the complement of a parallel
connection of two arrangements is diffeomorphic to a product of the
respective complements.  However, the diffeomorphism induces a 
nontrivial map on the abelianized fundamental groups, hence on
the parameter spaces for the homology jump loci.  The maps 
themselves also form a symmetric operad, and we use this 
observation as an organizational device in examining the 
effect of iterated parallel connections.

In particular, given an arrangement $\A=\set{H_1,\dots ,H_n}$ 
with multiplicity vector $m$, we define the {\em polarization}\/ 
$\A \pc m$ to be the iterated parallel connection of the form
$\A\plugin_{H_1}\Pl_{m_{H_1}}\plugin_{H_2} 
\cdots\plugin_{H_n}\Pl_{m_{H_n}}$, where $\Pl_k$ 
denotes a pencil of $k$ lines in $\C^2$, and $\plugin_{H}$
denotes a parallel connection that attaches a hyperplane of 
$\Pl_k$ to the hyperplane $H$ of $\A$.  Figure~\ref{fig:pc} 
shows an example.  The construction allows us to construct 
new arrangements with slightly customized characteristic varieties.

A crucial point here is the connection between the 
Milnor fiber of the (simple) arrangement $\A\pc m$ 
and the Milnor fiber of the multi-arrangement $(\A,m)$:  
the pullback of the cover $F(\A\pc m)\to U(\A\pc m)$ 
along the canonical inclusion $U(\A) \to U(\A\pc m) $ 
is equivalent to the cover $F(\A,m)\to U(\A)$.  Unlike $F(\A,m)$, 
the manifold $F(\A\pc m)$ is the Milnor fiber of a {\em reduced}\/ 
hypersurface.  Using this fact, we prove the following.

\begin{theorem}
\label{thm:intro polar tors}
Suppose $\A$ admits a pointed multinet, with distinguished 
hyperplane $H$ and multiplicity $m$.  Let $p$ be a prime 
dividing $m_H$. 
There is then a choice of multiplicities $m'$ on the deletion 
$\A' =\A\setminus \{H\}$ such that the Milnor fiber of the 
polarization $\B=\A' \pc m'$ has $p$-torsion in homology, 
in degree $1+\abs{\set{K\in \A':  m'_K\ge 3}}$.
\end{theorem}

The sufficient condition that the arrangement $\A$ must satisfy 
in order to yield a new arrangement $\B$ with torsion 
in $H_*(F(\B),\Z)$ is completely combinatorial:  if $L(\A)$ 
admits a certain kind of multinet, then everything else is 
predicted by that.

For instance, if $\A'$ is the deleted $\operatorname{B}_3$ 
arrangement as above, then the choice of multiplicities 
$m'=(8,1,3,3,5,5,1,1)$ produces an arrangement 
$\B=\A'\pc m'$ of $27$ hyperplanes in $\C^8$, such that 
$H_6(F(\B),\Z)$ has $2$-torsion of rank $108$. 
This is the smallest example known to us of an arrangement 
with torsion in the homology of its Milnor fiber.  We refer 
to \S\ref{subsec:conclude} for several questions that naturally 
arise from this work.

\section{Characteristic varieties and finite abelian covers} 
\label{sect:cv}

We start by reviewing the way in which the characteristic 
varieties of a space determine the homology of its regular, 
finite abelian covers, with coefficients in a field of 
characteristic not dividing the order of the cover.

\subsection{Characteristic varieties}
\label{subsec:char var}
Let $X$ be a connected, finite-type CW-complex, with fundamental 
group $G=\pi_1(X,x_0)$.  Throughout this paper, $\k$ will
denote an algebraically closed field.  Let $\G=\G_\k=\Hom(G,\k^*)$ 
be the affine algebraic group of $\k$-valued (multiplicative) characters 
on $G$.  Each character $\rho\colon G\to \k^*$ determines a rank-$1$ 
local system $\k_{\rho}$ on $X$: the $\k[G]$-module structure on $\k$ 
is given by $g\cdot a=\rho(g) a$. 

The {\em characteristic varieties}\/ of $X$ with coefficients in 
$\k$ are the jumping loci for homology with coefficients 
in rank-$1$ local systems on $X$:
\begin{equation}
\label{eq:cvs}
\V^q_d(X,\k)=\{\rho\in\G : \dim_{\k} H_q(X,\k_\rho)\ge d\}.
\end{equation}
It is readily seen that these loci are Zariski closed subsets
of $\G$, and depend only on the homotopy type of $X$. 
Moreover, the characteristic varieties are defined by 
determinantal conditions on integer equations;  
therefore $\V^q_d(X,\K)$ may be regarded as the 
$\K$-points of an affine scheme $\V^q_d(X)$ defined 
over $\Z$, for any extension $\K$ of the prime subfield.  

In each fixed degree $q\ge 0$, we have a descending filtration, 
$\widehat{G}=\V^q_0(X,\k)\supseteq \V^q_1(X,\k)\supseteq\cdots
\supseteq\V^q_d(X,\k)\supseteq\cdots$. 
Let $\V^q(X,\k)=\V^q_1(X,\k)$.
Note that $\V^0(X,\k)=\set{\bo}$, where $\bo$ is the identity of $\G$.  

As shown in \cite{PS10}, characteristic varieties behave well 
under products.  More precisely, let $X_1$ and $X_2$ be two 
connected, finite-type CW-complexes, and identify the character 
group $\widehat{\pi_1(X_1\times X_2)}$ with 
the product $\widehat{\pi_1(X_1)}\times \widehat{\pi_1(X_2)}$.
Then the (depth $1$) characteristic varieties of the 
product $X_1\times X_2$ are given by 
\begin{equation}
\label{eq:cvprod}
\V^q(X_1\times X_2,\k)=\bigcup_{i=0}^{q} 
\V^i(X_1,\k)\times \V^{q-i}(X_2,\k),
\end{equation}
for all $q\geq0$.

Given a finitely presented group $G$, we may define 
the characteristic varieties $\V^1_d(G,\k)$ as those 
of a (finite) presentation $2$-complex for $G$; the definition 
does not depend on the choice of presentation.  Moreover, 
the varieties $\V^1_d(G,\k)$ can be computed directly from 
such a presentation, by means of the Fox calculus.

\begin{example}
\label{ex:V of Fn}
The characteristic varieties of a free group can be computed 
by hand.  Let $F_n$ be the free group of rank $n$. 
Identifying $\widehat{F_n}=(\k^*)^n$, we have 
\begin{equation}
\label{eq:V of Fn}
\V^1_d(F_n,\k)=\begin{cases}
\widehat{F_n} &\text{for $n\geq2$ and $d\leq n-1$},\\
\set{\bo} & \text{for $d=n$},
\end{cases}
\end{equation}
and $V^q_d(F_n,\k)=\emptyset$, otherwise. 
\end{example}

\begin{example}
\label{ex:V of Fn star Zp}
Slightly more generally, consider a free product
$\Gamma=F_n * \Z_p$, where $\Z_p$ denotes the
cyclic group of order $p$.  In this case, $\widehat{\Gamma}=
\widehat{F_n}\times\widehat{\Z_p}$: note that $\widehat{\Z_p}$ has order
$p$ if $\ch(\k)\neq p$, and equals $\set{\bo}$ if $\ch(\k)=p$.

Let $H^\circ$ denote the identity component of an algebraic 
group $H$.  A direct computation shows that if $\ch(\k)\neq p$, then
\begin{equation}
\label{eq:v1freeprod}
\V^1(\Gamma,\k)=\begin{cases}
\widehat{\Gamma} & \text{if $n\ge 2$}, \\[2pt]
\big( \widehat{\Gamma}\setminus 
\widehat{\Gamma}^{\circ}\big) \cup  \set{\bo} 
& \text{if $n=1$ }, \\[2pt]
\end{cases}
\end{equation}
and $\V^q(\Gamma,\k)=\emptyset$ for $q\geq2$.  However, if 
$\ch(\k)=p$, then $\V^q(\Gamma,\k)=\widehat{\Gamma}$ for 
all $n\geq1$, and all $q\geq1$.
\end{example} 

The degree-$1$ characteristic varieties enjoy a nice 
functoriality property. 

\begin{lemma}[\cite{Su13a}]
\label{lem:epi cv}
If $\varphi\colon G \surj Q$ is an epimorphism, then, for each 
$d\ge 1$, the induced monomorphism between character groups, 
$\widehat{\varphi}_\k=\Hom(\varphi,\k^*)\colon \widehat{Q}_\k \inj \G_\k$, 
restricts to an embedding $\V^1_d(Q,\k) \inj \V^1_d(G,\k)$.
\end{lemma}

\subsection{Finite covers}
\label{subsec:finite covers}
As before, let $X$ be a connected, finite-type CW-complex.  
By the Universal Coefficients Theorem, 
\begin{equation}
\label{eq:uct bis}
\dim_{\C} H_i(X ,\C) \le \dim_{\k} H_i(X ,\k),
\end{equation}
for any field $\k$.  Moreover, if $\ch(\k)=p$ and the inequality 
is strict, then $H_i(X,\Z)$ has non-trivial $p$-torsion.

A well-known transfer argument shows that the homology groups 
of a finite cover of $X$ are ``larger" than the homology groups 
of $X$, at least away from any prime $p$ not dividing 
the order of the cover. The following lemma makes this fact 
more precise, in a way suited for our purposes.  

If $A$ is an abelian group, let $A_{(p)}$ denote its
$p$-primary part, obtained by localizing at the prime ideal 
$(p)\subseteq \Z$.

\begin{lemma}
\label{lem:transfer}
Let $\pi\colon Y\to X$ be an $r$-fold cover, and let 
$p$ be a prime not dividing $r$.   For each $i\ge 0$, 
the map $\pi_*\colon H_i(Y,\Z)_{(p)}\to  H_i(X,\Z)_{(p)}$ 
is surjective.
\end{lemma}

\begin{proof}
Let $\tau_*\colon H_i(X,\Z)_{(p)} \to H_i(Y,\Z)_{(p)}$ be the localization
of the transfer homomorphism.  Localization is exact, so 
$\pi_* \circ \tau_*\colon H_i(X,\Z)_{(p)}\to H_i(X,\Z)_{(p)}$ 
is multiplication by $r$.  Since $r$ is a unit in the local ring $\Z_{(p)}$, 
the composite is an isomorphism, so $\tau_*$ is surjective.
\end{proof}

In particular, if $H_i(X,\Z)$ has non-trivial $p$-torsion, then 
$H_i(Y,\Z)$ also does. The above argument also shows that 
the homomorphism $\pi_*\colon H_i(Y,\k) \to  H_i(X,\k)$ is 
surjective, for any field $\k$ of characteristic $0$ or $p$, with 
$p\nmid r$. 

\subsection{Finite abelian covers}
\label{subsec:abelian covers}
We now specialize to the case when the cover $Y\to X$ is 
a Galois cover, whose group of deck transformations is abelian. 
In this situation, a more precise formula for the homology groups 
of $Y$ can be given, again away from the primes dividing the order 
of the cover. 

Let $A$ be a finite abelian group.  A cohomology class $\chi\in H^1(X,A)$
may be regarded as a homomorphism $\chi\colon G\to A$, 
where $G=\pi_1(X,x_0)$.  Without much loss of generality, 
we may assume $\chi$ is surjective, in which case $\chi$  
determines a regular, connected $A$-cover of $X$, which 
we denote by $X^{\chi}$.  

Let $\k$ be an (algebraically closed) field, and let  
$\widehat{\chi}_\k\colon \widehat{A}_\k\to \widehat{G}_\k$ 
be the induced morphism between character groups. 
Then $\widehat{\chi}$ is injective, and its image is
(non-canonically) isomorphic to $A$.  If $\rho\colon G\to \k^*$ 
is a character belonging to $\im(\widehat{\chi}_\k)$, then 
$\rho=\iota\circ\chi$ for some homomorphism $\iota\colon A\to \k^*$; 
since $\chi$ is surjective, $\iota$ is unique, and we will denote it by
$\iota_\rho$.

The next proposition records a formula for the homology 
groups $H_q(X^{\chi}, \k)$, in the case when $\ch(\k)$ does 
not divide the order of $A$.  In the case $q=1$, this 
formula is well-known, and due to Libgober \cite{Li92}, 
Sakuma \cite{Sa95}, and Hironaka~\cite{Hir} 
in the case $\k=\C$, and to Matei--Suciu~\cite{MS} for 
other fields $\k$.   For the convenience of the reader, we 
include a self-contained proof, in this wider generality.  

\begin{theorem}
\label{thm:eko}
Let $X^{\chi}\to X$ be the regular cover defined by an epimorphism 
$\chi$ from $G=\pi_1(X)$ to a finite abelian group $A$. 
Let $\k$ be an algebraically closed field of characteristic $p$, 
where $p=0$ or $p\nmid \abs{A}$.  Then 
\begin{equation}
\label{eq:decomp}
H_q(X^{\chi},\k)=\bigoplus_{d\geq 1}\bigoplus_{\rho\in \im(\widehat{\chi}_\k) 
\cap  \V_d^q(X,\k)} \k_{\iota_\rho}, 
\end{equation}
as $\k[A]$-modules.  In particular, 
\begin{equation}
\label{eq:eko}
\dim_\k H_q(X^{\chi},\k)=\sum_{d\geq 1}\abs{\im(\widehat{\chi}_\k) 
\cap  \V_d^q(X,\k)}.
\end{equation}
\end{theorem}

\begin{proof}
The epimorphism $\chi$ gives the group algebra $\k[A]$ the structure of a
$(\k[G],\k[A])$-bimodule.  By Shapiro's Lemma, 
$H_q(X^{\chi},\k)\cong H_q(X,\k[A])$, as (right) $\k[A]$-modules, 
for each $q\ge 0$.

By our assumption on $\k$, the group algebra of $A$ is 
completely reducible, with one-dimensional, irreducible representations 
parametrized by $\widehat{A}_\k$.  As (left) $\k[G]$-modules, 
these are the image of
$\widehat{\chi}_\k\colon \widehat{A}_\k\to\widehat{G}_\k$, so
\begin{equation}
\label{eq:hqchi}
H_q(X, \k[A])\cong\bigoplus_{\rho\in\im(\widehat{\chi}_\k)} H_q(X,\k_\rho).
\end{equation}
For a fixed $q$ and $\rho$, let $b=\dim_\k H_q(X,\k_\rho)$.  By the remark
above, $H_q(X,\k_\rho)\cong(\k_{\iota_\rho})^{\oplus b}$ as a $\k[A]$-module.
By definition, $\rho \in \V^q_d(X,\k)$ if and only if 
$\dim_{\k} H_q(X,\k_{\rho}) \ge d$.  
This finishes the proof.
\end{proof}

Formula \eqref{eq:eko} can also be written as 
\begin{equation}
\label{eq:eko2}
\dim_\k H_q(X^{\chi},\k)=\dim_\k H_q(X,\k) + 
\sum_{d\geq 1}\abs{\im(\widehat{\chi}_\k) \cap  
(\V_d^q(X,\k)\setminus \set{\bo})}.
\end{equation}

\subsection{The cyclic case}
\label{subsec:cyclic covers}
To conclude this section, we further specialize to the 
case where $X^{\chi}\to X$ is a regular, finite cyclic cover.  

Let $A$ be a finite cyclic group, and let 
$\chi\colon G=\pi_1(X)\to A$ be a surjective homomorphism. 
Then $A$ acts on $X^{\chi}$ by deck transformations:
given the choice of a generator $\alpha$ of $A$, 
let $h=h_\alpha\colon X^{\chi}\to X^{\chi}$ be  
the monodromy automorphism, and let 
$h_*\colon H_q(X^{\chi},\k)\to H_q(X^{\chi},\k)$ be 
the induced map in homology.

\begin{corollary}
\label{cor:alg mono}
Assume $\ch(\k)\nmid \abs{A}$. Then, the characteristic polynomial of the 
algebraic monodromy, $\Delta^\k_q(t) = \det (t\cdot \id - h_*)$, 
is given by
\begin{equation}
\label{eq:charpoly}
\Delta_q^\k(t) =  \prod_{d\ge 1} 
\prod_{\rho\in \im(\widehat{\chi}_\k) 
\cap  \V_d^q(X,\k)}  (t-\iota_{\rho}(\alpha)). 
\end{equation}
\end{corollary}

\begin{proof}
Since $H_q(X^{\chi},\k)=H_q(X,\k[A])$ is a completely reducible
$\k[A]$-module, the automorphism $h_*$ is diagonalizable.  Its
eigenvalues, counted with multiplicity, are indexed by the $\k[A]$-irreducibles
appearing in the decomposition \eqref{eq:decomp} of Theorem \ref{thm:eko}.
The conclusion follows.
\end{proof}

\begin{remark}\label{rem:Galois}
Since each $\V^q_d(X,\k)$ is defined over $\Z$, 
the points of $\V^q_d(X,\overline{F})$ are stable under the
natural action of $\Gal(\overline{F}/F)$ on $\G_{\overline{F}}$, 
where $F$ denotes the prime subfield ($\Q$ or $\F_p$).  In particular,
the set $\im(\widehat{\chi}_\k)\cap  \V_d^q(X,\k)$ is closed under
the Galois action permuting $k$th roots of unity.  Therefore, 
$\Delta^\k_q(t)$ is a polynomial with integer coefficients, 
and does not depend on the choice of generator $\alpha$ 
for the cyclic group $A$.
\end{remark}

\subsection{Twisted Poincar\'{e} polynomial}
\label{subsec:poin}
For a given cyclic cover, we can package the above data in 
a generating function that depends only on the 
characteristic varieties of our space $X$. 

\begin{definition}
\label{def:alg mono gf}
Given a homomorphism $\chi$ from $\pi_1(X)$ to a finite 
cyclic group $A$, let $\Delta_{X,\chi}^\k$ be the polynomial in formal 
variables $x$ and $\mathbf{u}=\set{u_k\colon k\mid\abs{A}}$
given by
\begin{equation}
\label{eq:delta chi}
\Delta_{X,\chi}^\k(\mathbf{u},x) = \sum_{q\geq0,\, d\geq 1}
\sum_{\rho\in \im(\widehat{\chi}_\k)\cap  \V_d^q(X,\k)}
u_{\abs{\rho}} x^q,
\end{equation}
where $\abs{\rho}$ denotes the order of $\rho$ in 
the group $\widehat{G}$.  We note (as in Remark~\ref{rem:Galois}) that
$\Delta^\k_{X,\chi}(\mathbf{u},x)$ is a polynomial with integer coefficients.
\end{definition}

This notion can be recast in more familiar terms, as follows. 
Define the {\em twisted Poincar\'e polynomial}\/ of $X$ at a character 
$\rho\in \widehat{G}$ to be 
\begin{equation}
\label{eq:poin}
\poin^\k(X,\rho;x) = \sum_{q\geq0}\dim_\k H_q(X,\k_\rho) x^q.
\end{equation}

Note that $\poin^\k(X;x)=\poin^\k(X,\bo;x)$ is the usual 
Poincar\'e polynomial of $X$.  Now, in view of Theorem \ref{thm:eko}, 
we may also write
\begin{equation}\label{eq:sum of pp}
\Delta_{X,\chi}^\k(\mathbf{u},x) = \sum_{\rho\in\im(\widehat{\chi}_\k)}
u_{\abs{\rho}}\poin^\k(X,\rho;x).
\end{equation}

We note that specializing $u_k$ to 
$\frac{1}{\phi(k)}\log(\Phi_k(t))$ gives the formal power series 
\begin{equation}
\label{eq:fps}
\Delta_{X,\chi}^\k(t,x)=\sum_{q\geq0}\log\Delta^\k_q(t)x^q,
\end{equation}
while specializing each $u_k$ to $1$ gives the Poincar\'e polynomial of
the cover, $\poin^\k(X^{\chi},x)$.

\begin{example}  
\label{ex:free poin}
Let $X=\C \setminus \set{\text{$n$ points}}$, a classifying space 
for the free group $F_n$, $n\ge 2$.  Given a character 
$\rho\colon F_n \to \k^*$, the computation from 
Example~\ref{ex:V of Fn} yields  
\begin{equation}
\label{eq:poin free}
\poin^{\k}(X,\rho;x)=
\begin{cases}
(n-1)x & \text{if $\rho\ne \bo$},\\
1+nx & \text{if $\rho=\bo$}.
\end{cases}
\end{equation}
Now let $\chi\colon F_n \surj \Z_r$ be an 
epimorphism, and assume $\ch(\k)\nmid r$. Then
\begin{equation}
\label{eq:delta free}
\Delta_{X,\chi}^\k(\mathbf{u},x)=u_1(1+x)+ (n-1) \sum_{k\mid r} \phi(k)
u_k  x. \qedhere
\end{equation}
\end{example}

We conclude with a special case which will arise in \S\ref{sect:polar}.  
The next result rephrases the product formula \eqref{eq:cvprod} in a way 
that remembers the algebraic monodromy.  

\begin{proposition}
\label{prop:cover of product}
Suppose that $Y\to X_1\times X_2$ is a regular, cyclic cover defined by 
a surjective homomorphism $\chi\colon G_1\times G_2\to \Z_r$, 
for some $n>1$, where $G_i=\pi_1(X_i)$ for $i=1,2$.  Let $\k$ be an 
algebraically closed field for which $\ch(\k)\nmid r$.  
Then
\begin{equation}
\label{eq:delta prod}
\Delta_{X,\chi}^\k(\mathbf{u},x) = \sum_{\rho\in\im(\widehat{\chi}_\k)}
u_{\abs{\rho}}\poin^\k(X_1,\rho_1;x)\poin^\k(X_2,\rho_2;x),
\end{equation}
where $\rho_i$ denotes the projection of $\rho\in \widehat{G}$ to 
$\widehat{G_i}$. 
\end{proposition}

\begin{proof}
We use the formula \eqref{eq:sum of pp} together with the K\"unneth
formula for rank-$1$ coefficient modules, as in \cite[\S13.1]{PS10}.
\end{proof}

\section{Torsion in the homology of finite cyclic covers}
\label{sect:tors covers}

The characteristic varieties of a space $X$ depend on the characteristic 
of the ground field $\k$. In this section, we exploit this dependency 
to find integer torsion in the homology of regular, finite cyclic covers 
of $X$.  

\subsection{Varying the characteristic of the coefficient field}
\label{subsec:equiv}
As usual, let $X$ be a connected, finite-type CW-complex, 
with fundamental group $G=\pi_1(X,x_0)$.  

\begin{theorem}
\label{thm:compare fields}
Let $\chi\colon G\surj \Z_r$ be an epimorphism.  Suppose 
$\k$ is algebraically closed field of characteristic $p>0$, and 
$p\nmid r$.  Then, for each $q, d\ge 1$, 
\begin{equation}
\label{eq:uct}
\abs{\im(\widehat{\chi}_\C) \cap  \V_d^q(X,\C)}\leq 
\abs{\im(\widehat{\chi}_\k) \cap  \V_d^q(X,\k)}. 
\end{equation}
Moreover, if the inequality is strict, then $H_q(X^{\chi} ,\Z)$ has 
non-zero $p$-torsion.
\end{theorem}

\begin{proof}
For both $\K=\C$ and $\K=\k$, 
let $C_\K=\im(\widehat{\chi}_\K)$; our hypotheses
ensure that $C_\K$ is a cyclic group of order $r$.  
For each $k\mid r$, let $C^{(k)}_\K$ 
denote the set of elements of order $k$ in $C_\K$.  
Since characteristic varieties are closed under the Galois action 
fixing the prime subfield (see Remark~\ref{rem:Galois}), either
$C^{(k)}_\K\subseteq \V^q_d(X,\K)$, or $C^{(k)}_\K\cap \V^q_d(X,\K)=\emptyset$.
Set 
\begin{equation}
\label{eq:dk}
d^{(k)}_\K=\max\set{d \in \N : C^{(k)}_\K\subseteq\V^q_d(X,\K)}.
\end{equation}

Let $\Phi_k(t)\in\Z[t]$ denote the cyclotomic polynomial
of order $k$, and $\phi(k)=\deg\Phi_k(t)$.
Let $G$ act on $\Z[t]/\Phi_k(t)$ by multiplication by 
$t^{\chi(g)}$, for $g\in G$.  Then we have
\begin{align}
\rank_\Z H_q(X,\Z[t]/\Phi_k(t)) &=\dim_\C H_q\Big(X,\bigoplus_{\rho\in
C_\C^{(k)}}\C_\rho\Big) \label{eq:uct in proof}\\
&=\phi(k)d^{(k)}_\C.\nonumber
\end{align}
By the Universal Coefficients Theorem, the first equality in 
\eqref{eq:uct in proof}
becomes an inequality if $\C$ is replaced by $\k$.  Thus $d^{(k)}_\C\leq
d^{(k)}_\k$, for each $k$ dividing $r$.

Since  $\abs{\im(\widehat{\chi}_\K)\cap\V^q_d(X,\K)}=
\sum_{k\colon d^{(k)}_\K\geq d}\phi(k)$, 
for both $\K=\C$ and $\K=\k$, inequality \eqref{eq:uct} follows.

The last assertion follows at once from the UCT.
\end{proof}

Of course, inequality \eqref{eq:uct} is not always strict.  
A nice situation where equality holds is described next. 

\begin{proposition}
\label{prop:no torsion}
Suppose that, for a fixed $q\ge 1$, and for each $d\ge 1$, 
 there are subgroups $L_1, \dots , L_s$ of $H_1(X,\Z)$ such that 
 $\V_d^q(X,\k) = \bigcup_{i=1}^{s} \Hom(L_i, \k^*)$, 
for any algebraically closed field $\k$. In that case, for every 
homomorphism $\chi\colon G\surj \Z_r$, the polynomial 
$\Delta_{X,\chi}^\k(\mathbf{u},x)$ is independent of $\k$, 
and the group $H_q(X^{\chi} ,\Z)$ has 
no $p$-torsion, for any prime $p$ not dividing $r$.
\end{proposition}

\begin{proof}
Out assumption implies that the quantity 
$\abs{\im(\widehat{\chi}_\k) \cap  \V_d^q(X,\k)}$ 
is independent of $\k$ (for a fixed $\chi$).  The conclusions readily follow.
\end{proof}

The hypothesis of the proposition is satisfied for 
the space $X=\C \setminus \set{\text{$n$ points}}$ 
from Example \ref{ex:free poin}, and for a finite 
direct product of such spaces.

\subsection{Torsion in the homology of cyclic covers}
\label{subsec:ptors}

Inequality \eqref{eq:uct} above cannot be reversed, in general.  
We exploit this fact to detect torsion in the homology of certain 
cyclic covers, thereby proving Theorem \ref{thm:intro ptors1} 
from the Introduction.

\begin{theorem}
\label{thm:ptors}
Let $X^{\chi}\to X$ be a regular, finite cyclic cover, defined by an epimorphism 
$\chi\colon \pi_1(X)\surj \Z_r$.  Let $p$ be a prime not dividing $r$, and 
suppose that 
\begin{enumerate}
\item \label{chi1}
$\im(\widehat{\chi}_\C)\not\subseteq \V^q(X,\C)$, but
\item \label{chi2}
$\im(\widehat{\chi}_\k) \subseteq\V^q(X,\k)$,
\end{enumerate}
for some field $\k$ of characteristic $p$.  Then
$H_q(X^{\chi},\Z)$ has non-zero $p$-torsion.  

If, moreover,
$\im(\widehat{\chi}_\C) \cap \left( \V^q(G,\C)\setminus \set{\bo}\right) 
=\emptyset$, then
$H_q(X^{\chi},\Z)$ has $p$-torsion of rank at least $r-1$. 
\end{theorem}

\begin{proof}
The first claim follows immediately by combining \eqref{eq:eko} of
Theorem \ref{thm:eko} with \eqref{eq:uct} of 
Theorem~\ref{thm:compare fields}.

To prove the second assertion, using the assumption 
together with \eqref{eq:eko2}, we see that 
\[
\rank_\Z H_q(X^{\chi},\Z)=\dim_\C H_q(X^{\chi},\C)=\dim_\C H_q(X,\C).
\]
On the other hand, by \eqref{cv2}, we have that 
$\abs{\im(\widehat{\chi}_\k) \cap  \V^q_1(X,\k)}=r$. 
So by \eqref{eq:eko2}, 
\begin{align*}
\dim_\k H_q(X^{\chi},\k)&\geq \dim_\k H_q(X,\k)+r-1\\
&\geq \dim_\C H_q(X,\C)+r-1,
\end{align*}
and the conclusion follows by the UCT again.
\end{proof}

A particular case is worth mentioning. 

\begin{corollary}
\label{cor:torschi}
Let $G$ be a finitely generated group, and 
let $\chi\colon G\surj \Z_r$ be an epimorphism.  
Let $p$ be a prime not dividing $r$, and 
suppose that 
\begin{enumerate}
\item \label{cv1}
$\im(\widehat{\chi}_\C)\not\subseteq \V^1(G,\C)$, but
\item \label{cv2}
$\im(\widehat{\chi}_\k) \subseteq\V^1(G,\k)$.
\end{enumerate}
for some field $\k$ of characteristic $p$.  Then the abelianization of 
$N=\ker(\chi)$ has non-zero $p$-torsion.  If, moreover,
$\im(\widehat{\chi}_\C) \cap \left( \V^1(G,\C)\setminus \set{\bo}\right) 
=\emptyset$, then
$N_{\ab}$ has $p$-torsion of rank at least $r-1$. 
\end{corollary}

We shall see a concrete instance of this phenomenon (for $r=3$ 
and $p=2$) in Example \ref{ex:1torus}.

\subsection{Torsion in abelianized kernels}
\label{subsec:tors2} 

We are now ready to state the second main result of this 
section, which proves Theorem \ref{thm:intro ptors2} from 
the Introduction.

\begin{theorem}
\label{thm:tors1}
Let $G$ be a finitely generated group, and suppose there 
is an epimorphism $\varphi\colon G\surj \Z * \Z_p$, for 
some prime $p$.  Let  $\psi=\varphi\circ \pr_1\colon G\surj \Z$, 
where $\pr_1$ denotes the projection $\Z* \Z_p \surj \Z$, and 
suppose that $\im(\widehat{\psi}_\C)\not\subseteq\V^1(G,\C)$.
Then
\begin{enumerate}
\item \label{l1}
For all sufficiently large integers $r$ not divisible by $p$, 
there exists a normal subgroup $N\triangleleft G$ 
with $G/N\cong \Z_r$ and such that $N_{\ab}$ has 
non-zero $p$-torsion.

\item \label{l2}
If, in fact, $\im(\widehat{\psi}_\C)\cap\V^1(G,\C)
\subseteq \{\bo\}$, then $N_{\ab}$ has $p$-torsion of rank at least
$r-1$, for all $r>1$ coprime to $p$.
\end{enumerate}
\end{theorem}

\begin{proof}
To start with, we have the commuting triangle on the left 
side of diagram \eqref{eq:tr}, with $\Gamma=\Z*\Z_p$.  Applying 
the functor $\Hom(-,\C^*)$ gives, by Pontryagin duality, 
the triangle of injective morphisms in the upper right corner 
on the right side of \eqref{eq:tr}. 
\begin{equation}
\label{eq:tr}
\xymatrixcolsep{28pt}
\xymatrixrowsep{24pt}
\begin{gathered}
\xymatrix{ G \ar@{->>}^{\varphi}[r]  \ar@{->>}_(.4){\psi}[dr] 
& \Gamma \ar@{->>}^(.4){\pr_1}[d]\\
& \Z
}
\end{gathered}
\qquad   \leadsto \qquad 
\begin{gathered}
\xymatrix{\V^1(G,\C)  \ar@{^{(}->}[r]  & \widehat{G}_\C 
& \widehat{\Gamma}_\C \ar@{_{(}->}_{\widehat{\varphi}_\C}[l] 
\\
C  \ar@{^{(}->}[r]  & T  \ar@{_{(}->}[u]  
&  \C^* \ar@{_{(}->}[u] \ar@{_{(}->}[l] 
\ar@{_{(}->}_{\widehat{\psi}_\C}[ul] 
}
\end{gathered}
\end{equation}

The morphism $\widehat{\psi}_\C\colon \C^*\to\widehat{G}_\C$ 
is a one-parameter subgroup;  let $T$ be 
its image.   By hypothesis, $\V^1(G,\C)\cap T$ is a proper, 
closed subvariety of the algebraic torus $T$; since $\dim T=1$, 
this intersection is a finite set.
For sufficiently large $r$, then, $T$ contains a cyclic group $C$ of 
order $r$ for which $C\not\subseteq\V^1(G,\C)$;
without loss of generality, we may assume $p\nmid r$.

The situation so far is summarized on the right side of \eqref{eq:tr}. 
Applying the functor $\Hom_{\operatorname{alg}}(-,\C^*)$
gives the triangle of surjective homomorphisms on the right 
side of diagram \eqref{eq:tr2}, with $\bar\chi=\kappa\circ \bar\psi$.
Note that $\psi=\bar\psi\circ \ab$, where 
$\ab\colon G\to G_{\ab}$ is the abelianization map. 
Thus, if we set $\chi:=\bar\chi\circ \ab$, we have that 
$C=\im(\widehat{\chi}_\C)$.  Therefore, condition 
\eqref{cv1} from Corollary \ref{cor:torschi} is satisfied. 
\begin{equation}
\label{eq:tr2}
\begin{gathered}
\xymatrix{ 
G \ar@{->>}^(.4){\ab}[r]  \ar@/^1.5pc/@{->>}^{\psi}[rr]
\ar@{->>}_{\chi}[drr] 
& G_{\ab} \ar@{->>}^{\bar\psi}[r] \ar@{->>}^{\bar\chi}[dr]
& \Z  \ar@{->>}^(.45){\kappa}[d]
\\
&& \Z_r
}
\end{gathered}
\end{equation}

Next, let $\k$ be a field of characteristic $p$.  Applying the functor 
$\Hom(-,\k^*)$ yields the bottom triangle from diagram \eqref{eq:tr3}; 
clearly,  $\im(\widehat{\chi}_\k)\subset \im(\widehat{\psi}_\k)$. 
We also have $\im(\widehat{\psi}_\k) = \im (\widehat{\varphi}_\k)$, 
since $\Hom(\Z_p,\k^*) = \{\bo\}$, and so 
$\widehat{\Gamma}_\k \cong \k^*$. 
On the other hand, we know from Example~\ref{ex:V of Fn} that 
$\V^1(\Gamma,\k)=\widehat{\Gamma}_\k$.
By Lemma~\ref{lem:epi cv}, the morphism 
$\widehat{\varphi}_\k\colon \widehat{\Gamma}_\k \hookrightarrow 
\widehat{G}_\k$ restricts to an embedding 
$\V^1(\Gamma,\k)\hookrightarrow\V^1(G,\k)$.  

Putting things together, we infer that the map $\chi_\k$ factors 
through the dotted arrow, i.e., $\im(\widehat{\chi}_\k)\subset  
\V^1(G,\k)$; thus, condition \eqref{cv2} from 
Corollary \ref{cor:torschi} is also satisfied. 
\begin{equation}
\label{eq:tr3}
\begin{gathered}
\xymatrix{\V^1(G,\k)  \ar@{^{(}->}[r]  & \widehat{G}_\k 
& \widehat{\Gamma}_\k \ar@{_{(}->}_{\widehat{\varphi}_\k}[l] 
\\
& \Hom(\Z_r,\k^*)  \ar@{_{(}->}^{\widehat{\chi}_\k}[u]  
\ar@{^{(}->}^(.6){\widehat{\kappa}}[r] 
\ar@{_{(}-->}[ul] 
&  \k^* \ar_{\cong}[u]  
\ar@{_{(}->}_{\widehat{\psi}_\k}[ul] 
}
\end{gathered}
\end{equation}

Finally, set $N=\ker(\chi)$.  Applying Corollary \ref{cor:torschi}, 
we conclude that $N_{\ab}=H_1(N,\Z)$ has nonzero $p$-torsion

The last statement follows again from Corollary \ref{cor:torschi}. 
\end{proof}

\subsection{Discussion}
\label{subsec:ex disc} 
We conclude this section with a detailed example illustrating 
Theorem \ref{thm:tors1}, and a corollary showing how the 
hypothesis of the theorem can be tested more easily under 
a formality assumption.

\begin{example}
\label{ex:1torus}
Let $G=\langle x_1, x_2 \mid x_1 x_2^2 = x_2^2 x_1 \rangle$, and  
identify $\G=(\C^*)^2$, with coordinates $t_1$ and $t_2$. 
Taking the quotient of $G$ by the normal subgroup generated by 
$x_2^2$, we obtain an epimorphism $\varphi\colon G\surj \Z*\Z_2$. 
The morphism induced by $\varphi$ on character groups sends $\widehat{\Z}$ 
to $T=\{t_2=1\}$, and $\widehat{\Z}_2\setminus \{\bo\}$ to $\rho=(1,-1)$.  
By the above proposition, $\V^1(G)$ contains the translated 
torus $\rho T=\{ t_2=-1\}$.  Direct computation with Fox derivatives 
shows that, in fact, $\V^1(G)=\{1\}\cup \rho T$.  The inclusion of 
$\V^1(G)$ in the character torus is shown schematically in 
Figure~\ref{fig:torus}.

Now let $N$ be the kernel of the homomorphism $G\surj \Z_3$ 
sending $x_1\mapsto 1$ and $x_2\mapsto 0$.   By 
Theorem \ref{thm:eko} and Corollary \ref{cor:alg mono}, 
we have that $\dim_{\k} H_1(N,\k)=4$ and $\Delta^{\k}_1(t)=(t-1)^2(t^2+t+1)$ 
if $\ch(\k)=2$, whereas $\dim_{\k} H_1(N,\k)=2$ and $\Delta^{\k}_1(t)=(t-1)^2$,  
otherwise.  (In fact, direct computation shows that $H_1(N,\Z)=\Z^2\oplus \Z_2^2$.)
\end{example}

\begin{figure}[t]
\[
\begin{tikzpicture}[baseline=(current bounding box.center),scale=0.75]   
        \draw (-1,0) to[bend left] (1,0);
        \draw (-1.2,.1) to[bend right] (1.2,.1);
        \draw[rotate=0] (0,0) ellipse (100pt and 50pt); % identity component
        \draw[style=dotted] (0,-1.75) to[bend right] (0,-0.24);  
        \draw[style=loosely dotted] (0,-1.75) to[bend left] (0,-0.24);
        \draw (0,0.24) to[bend right] (0,1.75);  % translated component
        \draw[style=dashed] (0,0.24) to[bend left] (0,1.75);
        \node[circle,draw,inner sep=1pt,fill=black] at (0.2,-1) 
             [label=right:$\bo$] {};
        \node at (0,-1) [label=left:$T$] {};
        \node at (0,1) [label=left:$\rho T$] {};
        \node[circle,draw,inner sep=1pt,fill=blue] at (0.2,1) 
             [label=right:$\rho$] {};
        \node at (3,-1.2) [label=right:$/\C$] {};
\end{tikzpicture} \qquad\text{vs}\qquad
\begin{tikzpicture}[baseline=(current bounding box.center),scale=0.6]  
        \draw (-1,0) to[bend left] (1,0);
        \draw (-1.2,.1) to[bend right] (1.2,.1);
        \draw[rotate=0] (0,0) ellipse (100pt and 50pt); % identity component
        \draw[style=very thick] (0,-1.75) to[bend right] (0,-0.24);  
        \draw (0,-1.75) to[bend left] (0,-0.24);
        \node[circle,draw,inner sep=1pt,fill=black] at (0.2,-1) 
             [label=right:$\bo$] {};
        \node at (0,-1) [label=left:$T_\k$] {};
        \node at (3,-1.2) [label=right:$/\k$] {};
\end{tikzpicture}
\]
\caption{Characteristic varieties with translated torus components}
\label{fig:torus} 
\end{figure}
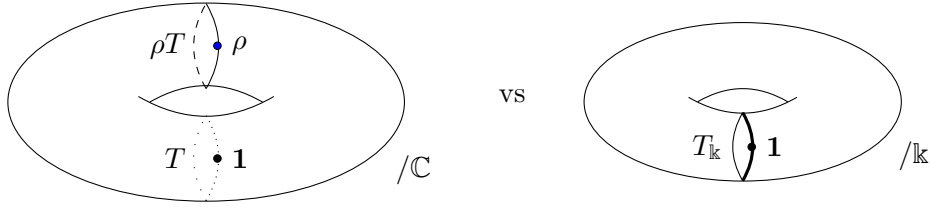

The (first) resonance variety of a finitely generated group $G$, 
denoted $\mathcal{R}^1(G,\C)$, is the set of all $a\in H^1(G,\C)$ 
for which there is an element $b\in H^1(G,\C)$ 
not proportional to $a$, such that $a\cup b=0\in H^2(G,\C)$.
The group $G$ is 
said to be $1$-formal if the complete, pro-nilpotent Lie algebra 
associated to $G$ is quadratic. As shown in \cite{DPS09}, if 
$G$ is $1$-formal, then both the tangent cone and the 
exponential tangent cone to the characteristic 
variety $\V^1(G,\C)$ coincide with $\mathcal{R}^1(G,\C)$.

\begin{corollary}
\label{cor:tors1}
Let $G$ be a $1$-formal group.  Suppose there is an epimorphism
$\varphi\colon G\surj \Z*\Z_p$ such that the kernel of the map 
$H^1(G,\Z)\to H^2(G,\Z)$ given by the cup product with the class 
$[\pr_1\circ \varphi] \in H^1(G,\Z)$ has rank $1$.
Then there are infinitely many integers $r$ (not divisible by $p$) for
which there exists an exact sequence $1\to N\to G\to \Z_r\to 1$ in which
$N_{\ab}$ has $p$-torsion.
\end{corollary}

\begin{proof}
Set $\psi=\pr_1\circ \varphi$.  The rank condition means that 
$[\psi]\otimes_\Z\C$ is not in $\mathcal{R}^1(G,\C)$, which 
implies (by $1$-formality) that 
the one-parameter subgroup $\im(\widehat{\psi})$
is not in the exponential tangent cone to $\V^1(G,\C)$.  Thus, 
$\im(\widehat{\psi})\not\subseteq\V^1(G,\C)$, and the conclusion 
follows from Theorem~\ref{thm:tors1}.
\end{proof}

\section{Quasi-projective varieties and orbifold fibrations} 
\label{sect:orbi}

We now turn to a geometric setting: that of smooth, quasi-projective varieties. 
The special nature of their characteristic varieties allows us to detect torsion 
in the homology of finite cyclic covers by means of ``small" orbifold fibrations. 

\subsection{Orbifold fundamental groups}
\label{subsec:orbi}

Let $\Sigma_{g,r}$ be a Riemann surface 
of genus $g\ge 0$, with $r\ge 0$ points removed. 
Fix points $q_1,\dots ,q_s$ on the surface, and 
assign to these points integer weights 
$\mu_1,\dots , \mu_s$ with $\mu_i\ge 2$. The resulting object, 
$\Sigma=(\Sigma_{g,r},\mu)$, is called a ($2$-dimensional) 
{\em orbifold}.  

Set $\abs{\mu}=s$, and write $n=2g+r-1$. We say 
that an orbifold $\Sigma$ as above is 
{\em hyperbolic}\/ if its orbifold Euler characteristic, 
\begin{equation}
\label{eq:chiorb}
\chi^{\orb}(\Sigma)= 
2-2g-r-\sum_{i=1}^{s} (1-1/\mu_i),
\end{equation}
is negative.  Furthermore, we say that a hyperbolic 
orbifold $\Sigma$ is {\em small}\/ if either 
$\Sigma=S^1\times S^1$ and $\abs{\mu} \geq2$, or 
$\Sigma=\C^*$ and $\abs{\mu} \geq1$; 
otherwise, we say the orbifold is {\em large}.

The orbifold fundamental 
group $\Gamma=\pi_1^{\orb}(\Sigma_{g,r}, \mu)$ 
is given by
\begin{equation}
\begin{aligned}
\label{eq:orbipi1}
\Gamma&=
\left\langle 
\begin{array}{c}
x_1,\dots, x_g, y_1,\dots , y_g, \\[2pt]
z_1, \dots ,z_s 
\end{array}\left|
\begin{array}{c}
[x_1,y_1]\cdots [x_g,y_g] z_1\cdots z_s =1, \\[2pt]
z_1^{\mu_1}=\cdots =z_t^{\mu_s}=1
\end{array}
\right.
\right\rangle 
&\text{if $r=0$,}\\
\Gamma&=
F_{n} * \Z_{\mu_1} * \cdots * \Z_{\mu_s}
&\text{if $r>0$.}
\end{aligned}
\end{equation}

Note that  $\Gamma_{\ab} = \Z^{2g} \oplus A$, 
where $A=\Z_{\mu_1}\oplus \cdots \oplus \Z_{\mu_s}/(1,\dots ,1)$ 
in the first case, and $\Gamma_{\ab}=\Z^{n}\oplus A$, 
where $A=\Z_{\mu_1}\oplus \cdots \oplus \Z_{\mu_s}$ 
in the second case.  In either case, 
identify $\widehat{\Gamma}=\widehat{\Gamma}^{\circ}  
\times \widehat{A}$.  A Fox calculus computation as in 
\cite{ACM13} shows that
\begin{equation}
\label{eq:v1piorb}
\V^1(\Gamma,\C)=\begin{cases}
\widehat{\Gamma} & 
\text{if $\Sigma$ is a large hyperbolic orbifold},
\\[2pt]
\big( \widehat{\Gamma}\setminus 
\widehat{\Gamma}^{\circ}\big) \cup  \{1\} 
& \text{if $\Sigma$ is a small hyperbolic orbifold},
\end{cases}
\end{equation}
and is a finite set of torsion characters, otherwise.

\subsection{Orbifold pencils}
\label{subsec:orbipen}

Now let $X$ be a smooth, connected quasi-projective 
variety, with fundamental group $G=\pi_1(X,x_0)$.  Let 
$(\Sigma_{g,r},\mu)$ be a $2$-dimensional 
orbifold with marked points $(q_1,\mu_1),\dots, (q_s,\mu_s)$. 
A surjective, holomorphic map $f\colon X\to \Sigma_{g,r}$ 
is called an {\em orbifold fibration}\/ (or, a pencil) if the 
generic fiber is connected, the multiplicity of the fiber 
over each point $q_i$ equals $\mu_i$, and $f$ admits an 
extension $\bar{f}\colon \overline{X}\to \Sigma_{g}$ to 
compactifications which is also a surjective, holomorphic 
map with connected generic fibers. 

We say that a hyperbolic orbifold fibration 
$f\colon X\to (\Sigma_{g,r},\mu)$ 
is either large or small according to its base.  
If $b_1(\overline{X})=0$, then an easy argument with 
mixed Hodge structures shows that $g=0$.

Any pencil $f$ as above 
induces an epimorphism $f_{\sharp} \colon G \surj \Gamma$, 
where $\Gamma=\pi_1^{\orb}(\Sigma_{g,r}, \mu)$, 
and thus a monomorphism 
$\widehat{f_{\sharp}}\colon \widehat{\Gamma}\inj \widehat{G}$. 
By Lemma \ref{lem:epi cv} and the computation from 
\eqref{eq:v1piorb}, we see that $\widehat{f_{\sharp}}$ sends 
$\V^1(\Gamma,\C)$ to a union of (possibly torsion-translated) 
subtori inside $\V^1(X,\C)$.  In fact, a more precise result  
holds.

\begin{theorem}[\cite{Ar, Di07, ACM13}]
\label{thm:arapura bis}
Let $X$ be a smooth, quasi-projective variety.  Then
\[
\V^1(X,\C)=\bigcup_{\text{$f$ large}} \im(\widehat{f_{\sharp}}) \cup 
\bigcup_{\text{$f$ small}} \Big( \im(\widehat{f_{\sharp}}) \setminus 
\im(\widehat{f_{\sharp}})^{\circ} \Big) \cup Z, 
\]
where $Z$ is a finite set of torsion characters. 
\end{theorem}

In particular, every positive-dimensional component of $\V^1(X,\C)$ 
is of the form $\rho\cdot T$, where $T$ is an algebraic subtorus 
in $H^1(X,\C^{*})$, and $\rho$ is of finite order (modulo $T$).  
The component arises from an orbifold fibration with base 
$\Sigma_{g,r}$, so that $T$ has dimension $n:=b_1(\Sigma_{g,r})$.  
We will call $T$ the {\em direction torus}\/ associated with the
orbifold fibration: see Figure~\ref{fig:torus}.

In the case where $\rho T=T$, we must have $n=2g\ge 4$ 
or $n=2g+r-1\ge 2$, according to whether the number of 
punctures $r=0$ or not.   On the other hand, if $\rho\notin T$,  
then the direction torus $T$ itself is also a component of 
$\V^1(X,\C)$ unless $n=2$ (if $r=0$) or $n=1$ (if $r>0$).

\begin{figure}
\end{figure}

\begin{theorem}[\cite{DPS08,ACM13}]
\label{thm:dirtor}
For any smooth, quasi-projective variety $X$, 
the direction tori associated with two 
orbifold fibrations of $\V^1(X,\C)$ either coincide or 
intersect only in finitely many points.
\end{theorem}

\subsection{Torsion in cyclic covers}
\label{subsec:tors flat}

We are now ready to state and prove the main result of this 
section, which establishes Theorem \ref{thm:intro pencil} 
from the Introduction. 

\begin{theorem}
\label{thm:flat pencil}
Let $X$ be a smooth, quasi-projective variety.  Suppose there 
is a small orbifold fibration $f\colon X\to (\Sigma,(\mu_1,\ldots,\mu_s))$ 
and a prime $p$ dividing $\gcd\set{\mu_1,\ldots, \mu_s}$.  
Then, for all sufficiently large integers $r>1$ not divisible by $p$, 
there exists a regular, $r$-fold cyclic cover $Y\to X$ 
such that $H_1(Y,\Z)$ has $p$-torsion.  

If, moreover, $\V^1(X,\C)$ 
contains no zero-dimensional, nonidentity components, then $H_1(Y,\Z)$ has
$p$-torsion of rank at least $r-1$ for all $r>1$ not divisible by $p$.
\end{theorem}

\begin{proof}
The orbifold fibration $f$ induces an epimorphism 
$f_\sharp\colon\pi_1(X)\to\pi_1^{\orb}(\Sigma,\mu)$. 
Since $p\mid \gcd\set{\mu_1,\ldots, \mu_s}$, 
the target group surjects onto $\Z * \Z_p$.  
Thus, we obtain an epimorphism $\varphi\colon\pi_1(X)\surj 
\Z * \Z_p$.  

Let $\psi=\varphi\circ \pr_1\colon \pi_1(X)\surj \Z$.  
Clearly, $T:=\im(\widehat{\psi})$ is the direction torus associated 
to $f$.  Let $\rho T'$ be a positive-dimensional component of 
$\V^1(X,\C)$. By Theorem \ref{thm:dirtor}, $T\cap T'$ is finite; thus, 
$T\cap \rho T'$ is also finite.  Hence, $T\not\subseteq\V^1(X,\C)$, 
and so the first conclusion follows by Theorem~\ref{thm:tors1}.

Note that zero-dimensional components of $\V^1(X,\C)$ 
are necessarily torsion points, by \cite{ACM13}.  By 
Theorem~\ref{thm:arapura bis}, then, if all nonidentity
components of $\V^1(X,\C)$ are positive-dimensional, the 
intersection $T\cap\V^1(X,\C)$ is either empty or $\set{\bo}$.  
The second conclusion follows by Theorem~\ref{thm:tors1}\eqref{l2}.
\end{proof}

\begin{example}
\label{ex:ccm}
Following \cite{CCM}, 
let $C_1$ be a (smooth, complex) curve of genus $2$
with an elliptic involution $\sigma_1$ and $C_2$ a curve of
genus $3$ with a free involution $\sigma_2$.
The group $\Z_2$ acts freely on the product $C_1 \times C_2$
via the involution $\sigma_1\times \sigma_2$; let
$M$ be the quotient surface.  The projection 
$C_1 \times C_2\to C_1$  descends to an orbifold
fibration $f_1 \colon M\to \Sigma_1=C_1/\sigma_1$ with two 
multiple fibers, each of multiplicity $2$, while the projection 
$C_1 \times C_2\to C_2$ descends to a holomorphic fibration 
$f_2\colon M\to \Sigma_2=C_2/\sigma_2$.

It is readily seen that $H=H_1(M,\Z)$ is isomorphic to $\Z^6$;
fix a basis $e_1,\dots , e_6$ for this group.  A computation 
detailed in \cite{Su13a} shows that the variety $\V^1(M,\C)\subset (\C^*)^6$
has precisely two components, arising from the orbifold fibrations 
described above. The first component is the subtorus 
$T_1=\{(1,1,t_3,t_4,t_5,t_6)\in(\C^*)^6\}$, which actually 
lies in $\V^1_4(M,\C)$, while the second component 
is the translated subtorus 
$\rho T_2=\{(t_1,t_2,-1,1,1,1)\in(\C^*)^6\}$, which 
lies in $\V^1_2(M,\C)$.

Let $r>1$ be an odd integer, and choose any epimorphism
$\chi\colon H_1(M,\Z)\surj\Z_r$ for which $\chi(e_i)=0$ for $3\leq i\leq 6$.
Then the cyclic group $\im(\widehat{\chi})$ is contained in $T_2$, yet it
intersects $\V^1(M,\C)$ only at the trivial character $\bo$.  Hence,  
by Theorem~\ref{thm:tors1}, there is $2$-torsion in $H_1(M^\chi,\Z)$.
For instance, if $r=3$ and $\chi=(1,1,0,0,0,0)$, then 
$H_1(M^\chi,\Z)=\Z^6\oplus \Z_2^4$, and the characteristic 
polynomial of the algebraic monodromy is 
$\Delta^{\k}_1(t)=(t-1)^6(t^2+t+1)^2$ if $\ch(\k)=2$, and 
$\Delta^{\k}_1(t)=(t-1)^6$, otherwise.
\end{example}

\section{Hyperplane arrangements, multinets, and deletions} 
\label{sect:arr}

In this section, we further specialize to the case of complex 
hyperplane arrangements.  The first characteristic 
variety of the complement of such an arrangement 
is determined in large part by a combinatorial construction, 
known as a ``multinet". Suitable deletions, then, produce 
small pencils, which will be put to use later on.

\subsection{The complement of an arrangement}
\label{subsec:hyp arr}

Let $\A$ be a finite set of hyperplanes 
in some finite-dimensional complex vector space $\C^{\ell}$.  
Most of the time, we will assume the arrangement is central, 
i.e., all hyperplanes pass through the origin.  In this case, 
the product 
\begin{equation}
\label{eq:defpoly}
Q(\A)=\prod_{H\in \A} f_H,
\end{equation}
is a defining polynomial for the arrangement, where 
$f_H\colon \C^{\ell} \to \C$ are linear forms for which $\ker(f_H)=H$. 

At times, we may wish to allow multiplicities on the hyperplanes: 
if $m\in\Z^\A$ is a lattice vector with $m_H\geq1$ 
for each $H\in\A$, the pair $(\A,m)$ is called a {\em multiarrangement}.
A defining polynomial for the multiarrangement is the 
product  
\begin{equation}
\label{eq:qam}
Q(\A,m)=\prod_{H\in \A} f_H^{m_H}.
\end{equation}

Let $M(\A)=\C^{\ell}\setminus \bigcup_{H\in \A} H$ be the 
complement of the arrangement.  Let $U(\A)=\P M(\A)$ denote the 
image of $M(\A)$ in $\P^{\ell-1}$: it is well-known that
$M(\A)\cong U(\A)\times \C^*$.  

Fix an ordering of the hyperplanes, and let 
$n=\abs{\A}$ be the cardinality of $\A$. 
A standard construction linearly embeds $M(\A)$ in a complex 
algebraic torus: let $\iota\colon \C^{\ell}\to\C^n$ be the map given by
\begin{equation}
\label{eq:torusembedding}
\iota(x)=(f_H(x))_{H\in\A}.
\end{equation}
Clearly, this map restricts to an inclusion $\iota\colon M(\A)\inj (\C^*)^n$.  
Since $\iota$ is equivariant with respect to the diagonal action of $\C^*$ 
on both source and target, it descends to a map 
$\overline{\iota}\colon U(\A)\inj (\C^*)^n/\C^*$.
We note the following for future reference. 

\begin{lemma}
\label{lem:H1ofA}
The map $\iota\colon M(\A)\to(\C^*)^n$ is a classifying map 
for the universal abelian cover of $M(\A)$: that is, 
$\iota_{\#}\colon\pi_1(M(\A))\to \Z^n$ is the natural 
projection $\pi_1(M(\A))\twoheadrightarrow \pi_1(M(\A))_{\ab}$.
Similarly, $\overline{\iota}_{\#}\colon\pi_1(U(\A))\to \Z^n/(1,\ldots,1)$ 
is the abelianization map for the fundamental group of the 
projective complement.
\end{lemma}

For each $H\in \A$, let $x_H$ denote the based 
homotopy class of a compatibly oriented meridian curve 
about the hyperplane $H$; as is well-known, the group 
$\pi_1(M(\A))$ is generated by these elements. 
By a slight abuse of notation, we will denote the 
image of $x_H$ in $H_1(M(\A),\Z)$ by the same symbol. 
Similarly, we will denote by $\overline{x}_H$ the image of 
$x_H$ in both $\pi_1(U(\A))$ and its abelianization. 

It follows from Lemma~\ref{lem:H1ofA} that $H_1(M(\A),\Z)\cong\Z^n$, 
with basis $\set{x_H\colon H\in\A}$, and 
$H_1(U(\A),\Z)\cong\Z^n \slash \big(\sum_{K\in\A} x_K\big)$. 

\subsection{Characteristic varieties of arrangements}
\label{subsec:cv arr}

Using the standard meridian basis of $H_1(M(\A),\Z)=\Z^n$, 
we may identify the character group $\Hom(\pi_1(M(\A)),\k^*)$ 
with the torus $(\k^*)^n$.  From the product 
formula \eqref{eq:cvprod}, we see that $\V^q_1(M(\A),\k)$ 
is isomorphic to $\V^q_1(U(\A),\k) \cup  \V^{q-1}_1(U(\A),\k)$.  
Thus, we may view the characteristic varieties $\V^q_1(M(\A),\k)$ 
as subvarieties of the algebraic torus 
$\set{t\in (\k^*)^n\mid t_1\cdots t_n=1}\cong (\k^*)^{n-1}$. 

We will be only interested here in the degree $1$ characteristic 
variety, $\V^1(\A,\k):=\V^1_1(M(\A),\k)=\V^1_1(U(\A),\k)$.   
To compute these varieties, we may assume, without loss 
of generality,  that $\A$ is a central arrangement in $\C^3$.

If $\B\subset \A$ is a sub-arrangement, the inclusion 
$U(\A) \inj U(\B)$ induces an epimorphism 
$\pi_1(U(\A)) \surj \pi_1(U(\B))$.  
By Lemma \ref{lem:epi cv}, the resulting monomorphism 
between character groups restricts to an embedding 
$\V^1(\B,\k) \inj \V^1(\A,\k)$.  Components of $\V^1(\A,\k)$ 
which are not supported on any proper sub-arrangement 
are said to be essential.  

The hyperplane complement $M(\A)$ is a smooth, quasi-projective 
variety.  By work of Arapura \cite{Ar}, the complex characteristic varieties   
$\V^q_1(M(\A),\C)$ are unions of unitary translates of algebraic 
subtori in $(\C^{*})^n$. By Theorem \ref{thm:arapura bis}, the 
variety $\V^1(\A)=\V^1(\A,\C)$ is, in fact, a union 
of torsion-translated subtori.  

\subsection{Multinets}
\label{subsec:multinets}
 
All the irreducible components of $\V^1(\A)$ passing through the 
origin can be described in terms of the existence of a multinet on 
the hyperplanes of $\A$, a notion we now recall. 

\begin{definition}
\label{def:multinet}
A {\em multinet}\/ on an arrangement $\A$ is a partition of $\A$ 
into $k\geq3$ subsets $\A_1,\ldots,\A_k$, together with an assignment 
of multiplicities, $m\colon \A\to \N$, and a subset 
$\XX\subseteq L_2(\A)$, called the base locus, 
such that:
\begin{enumerate}
\item $\sum_{H\in\A_i} m_H=d$, independent of $i$.
\item For each $H\in\A_i$ and $H'\in \A_j$ with $i\neq j$, the flat 
$H\cap H'$ belongs to $\XX$.
\item For each $X\in\XX$, the sum $n_X:=\sum_{H\in\A_i\colon H\leq X} m_H$
is independent of $i$.
\item For each $1\leq i\leq k$ and $H,H'\in\A_i$, there is a sequence
$H=H_0,H_1,\ldots, H_r=H'$ such that $H_{j-1}\cap H_j\not\in\XX$ for
$1\leq j\leq r$.
\end{enumerate}
\end{definition}

We say that a multinet  as above has $k$ classes and weight $d$, 
abbreviated as a $(k,d)$-multinet.  Without loss of generality, we 
may assume that $\gcd\set{m_H\colon H\in\A}=1$.  If all multiplicities 
$m_H$ are equal to $1$, we say the multinet is {\em reduced}. 
If, furthermore, every flat in $\XX$  is contained in precisely one 
hyperplane from each class, then the multinet is called 
a {\em $(k,d)$-net}. 

There is a somewhat trivial construction of multinets: set 
$\XX=\set{X}$ for any flat $X\in L_2(\A)$ that lies on at 
least three lines, and form a net on the subarrangement of 
hyperplanes containing $X$ by assigning each hyperplane 
multiplicity $1$ and putting one hyperplane in each class.
However, the existence of multinets for which $\abs{\XX}>1$ is much 
more subtle.  We distill some of the known results, as 
follows.

\begin{theorem}[\cite{PY}, \cite{Yu09}] 
\label{thm:pyy}
Let $(\A_1,\dots,\A_k)$ be a multinet on $\A$, with multiplicity 
vector $m$ and base locus $\XX$. 
\begin{enumerate}
\item \label{multi1}
If $\abs{\XX}>1$, then $k=3$ or $4$.    %% PY
\item \label{multi2}
If there is an $H\in \A$ such that $m_H>1$, then $k=3$. %% Yu09
\end{enumerate}
\end{theorem}

Although infinite families of multinets are 
known with $k=3$, only one is known to exist for $k=4$: 
the $(4,3)$-net on the Hessian arrangement. 

\subsection{Pencils, multinets, and translated tori}
\label{subsec:arr pencils}
Recall we are assuming $\A$ is a central arrangement in $\C^3$. 
As shown by Falk, Pereira, and Yuzvinsky in \cite{FY, PY}, 
every $(k,d)$-multinet on $(\A,m)$ determines a large orbifold fibration 
$f_m\colon M(\A)\to\Sigma_{0,k}$, which in turn produces a 
$(k-1)$-dimensional, connected component of $\V^1(\A)$ 
containing $\bo$.  

If $\abs{\XX}=1$, then the trivial net described above gives 
rise to a {\em local}\/ component of $\V^1(\A)$.   
Otherwise, let $Q_i=\prod_{H\in\A_i}f_H^{m_H}$,
for $1\leq i\leq k$, so that the polynomial $Q(\A,m)$ factors as 
$Q_1\cdots Q_k$.  Define a rational map $f_m\colon \C^3\to\P^1$ by 
\begin{equation}
\label{eq:FYpencil}
f_m(x)=[Q_1(x):Q_2(x)].
\end{equation}

Clearly the restriction to the  complement, 
$f_m\colon M(\A)\to\P^1$, is a regular map, and the 
points $[0\co 1]$ and $[1\co 0]$ are not in its image.
Along the same lines, if $[a:b]\not\in f_m(M(\A))$, then the polynomial
$aQ_2-bQ_1$ does not vanish on $M(\A)$, so it is a product of linear forms.
More is true: it turns out that there exist $k$ distinct pairs
$\set{(a_i,b_i)\in\C^2\colon 1\leq i\leq k}$
for which $Q_i=a_iQ_2-b_iQ_1$, for each $i$, and 
\begin{equation}
\label{eq:imfm}
\im(f_m)=\P^1\setminus \set{[a_i\co b_i]\colon 1\leq i\leq k}
\cong \Sigma_{0,k}.
\end{equation}
(We also see that the choice of order of classes affects $f_m$ only
by a linear automorphism of $\P^1$.)

In general, though, the characteristic variety $\V^1(\A)$ has  
irreducible components not passing through the the origin, as we 
shall see in the next subsection.

Following \cite{Su13a}, let us summarize the above discussion, 
as follows. As before, let $\A$ be an arrangement of $n$ hyperplanes.   

\begin{theorem}
\label{thm:cv arr}
Each irreducible component of $\V^1(\A)$ is a 
torsion-translated subtorus of $(\C^{*})^{n}$. 
Moreover, each positive-dimensional, non-local 
component is of the form $\rho \cdot T$, where $\rho$ is 
a torsion character, 
$T=\widehat{f_{\sharp}} \big(\widehat{\pi_1(\Sigma_{0,r})}\big)$, 
for some orbifold fibration $f\colon M(\A)\to (\Sigma_{0,r},\mu)$, 
and  either 
\begin{enumerate}
\item \label{tt1} $r=2$, and $f$ has at least one multiple fiber, or

\item \label{tt23}  $r=3$ or $4$,  and there exists a multiplicity
vector $m$ and multinet with $r$ classes on $(\A,m)$ for which $f=f_m$.
If $r=4$, moreover, $m_H=1$ for each $H$.
\end{enumerate}
\end{theorem}

\subsection{Deletions and pencils}
\label{subsec:deletion}

Fix now a hyperplane $H\in \A$.  The arrangement  
$\A'=\A\setminus \{H\}$ is then called the {\em deletion}\/ 
of $\A$ with respect to $H$.  Before proceeding, let us 
introduce a variation on Definition~\ref{def:multinet}. 

\begin{definition}
\label{def:pointed multinet}
A {\em pointed multinet}\/  on $\A$ is a multinet structure 
$((\A_1,\dots, \A_k), m, \XX)$, together with a distinguished 
hyperplane $H\in \A$ for which $m_H>1$, and $m_H \mid n_X$ 
for each flat $X\ge H$ in the base locus.
\end{definition}

\begin{proposition}
\label{prop:del}
Suppose $\A$ admits a pointed multinet, and $\A'$ is 
obtained from $\A$ by deleting the distinguished hyperplane 
$H$.  Then $U(\A')$ supports a small pencil, and 
$\V^1(\A')$ has a component which is a $1$-dimensional 
subtorus, translated by a character of order $m_H$.
\end{proposition}

\begin{proof}
Without loss of generality, assume that $H\in\A_1$, so that $f_H\mid Q_1$.
Consider the regular map given
by \eqref{eq:FYpencil}, $f_m\colon M(\A)\to \P^1$. 
Since $f_H$ does not divide $Q_2$, we may extend 
$f_m$ to a regular map $\overline{f}_m\colon M(\A')\to\P^1$.  
By construction, $\im(\overline{f}_m) \setminus \im(f_m)=\set{[0\co 1]}$.  
The image of $f_m$ is given by \eqref{eq:imfm}, so the image of 
$\overline{f}_m$ equals $\P^1\setminus\set{[1\co 0],[a_3\co b_3]}$.

By the pointed multinet hypothesis, the degrees of the restrictions 
of $Q_1(x)$ and $Q_2(x)$ to the hyperplane $H$
are both divisible by $m_H$.  It follows that the corestriction of 
$\overline{f}_m\colon M(\A')\to\P^1$ to its image is actually an orbifold 
fibration onto $(\Sigma_{0,2},m_H)$.  Hence, $\overline{f}_m$ is 
the desired small pencil, with the special fiber over $[0\co 1]$. 
The last statement now follows from Theorem \ref{thm:cv arr}.
\end{proof}

\begin{figure}
\begin{tikzpicture}[scale=0.85]
\draw[style=thick] (0,0) circle (2.4);
\node at (-2,0.25) {$2$};
\node at (0.25,-2) {$2$};
\node at (2.6,0.5) {$2$};
\clip (0,0) circle (2);
\draw[style=thick,densely dashed,color=blue] (-1,-2) -- (-1,2);
\draw[style=thick,densely dotted,color=red] (0,-2) -- (0,2);
\draw[style=thick,densely dashed,color=blue] (1,-2) -- (1,2);
\draw[style=thick,densely dotted,color=red] (-2,-1) -- (2,-1);
\draw[style=thick,densely dashed,color=blue] (-2,0) -- (2,0);
\draw[style=thick,densely dotted,color=red] (-2,1) -- (2,1);
\draw (-2,-2) -- (2,2);
\draw (-2,2) -- (2,-2);
\end{tikzpicture}
\caption{A $(3,4)$-multinet on the ${\rm B}_3$ arrangement}
\label{fig:B3}
\end{figure}
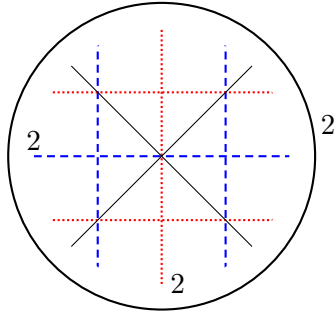

We illustrate this procedure with a concrete example.

\begin{example}
\label{ex:deleted B3}
Let $\A$ be reflection arrangement of type $\operatorname{B}_3$, 
with defining polynomial $Q=xyz(x-y)(x+y)(x-z)(x+z)(y-z)(y+z)$. 
This arrangement supports the multinet $m$ depicted in 
Figure \ref{fig:B3}.  

Let $\A'= \A \setminus \{z=0\}$ be the so-called 
deleted $\operatorname{B}_3$ arrangement, 
and order its hyperplanes in the same way  
as the factors of $Q$, above.
As noted in \cite{Su02}, the characteristic variety 
$\V^1(\A')\subset (\C^{*})^8$ contains a component 
of the form $\rho T$, where 
\[
\rho=(1,1,-1,-1,-1,-1,1,1) \quad\text{and}\quad 
T=\set{(t^2,t^{-2},1,1,t^{-1},t^{-1},t,t) \mid t\in \C^{*}}.
\]

To see how Proposition~\ref{prop:del} recovers this translated subtorus,
consider the polynomials $Q_1=z^2(x-y)(x+y)$, $Q_2=y^2(x-z)(x+z)$, 
and $Q_3=x^2(y-z)(y+z)$ defined by the classes of the multinet $m$ 
supported on $\A$, and let $f_m\colon M(\A)\to\P^1$ be the map given by
\[
f_m(x,y,z)=[z^2(x-y)(x+y):y^2(x-z)(x+z)].
\]

The point $[0\co 1]$ is not in the image of  $f_m$; however, extending 
$f_m\colon M(\A')\to\P^1$ by allowing $z=0$ extends the image of 
$f_m$ to include $[0\co 1]$.  Moreover, for $z=0$, we have 
$f_m=[0\co y^2x^2]$, so the fiber over $[0\co 1]$ has multiplicity $2$.
\end{example}

\section{Cyclic covers and Milnor fibrations of arrangements} 
\label{sect:cyclic arr}

In this section, we turn to the Milnor fibration of a multiarrangement.  
As is well-known, the Milnor fiber can be realized as a finite cyclic 
cover of the projectivized complement. Our techniques, then, 
permit us to find torsion in the first homology of the 
Milnor fiber, provided certain combinatorial constraints 
are satisfied.

\subsection{Cyclic covers of arrangement complements}
\label{subsec:hom arr}

Let $\A$ be a (nonempty) central arrangement, with complement 
$M(\A)$.  The homology groups of $M(\A)$ are well-known to be 
torsion-free.
However, combining the results above, we can now give a rather
delicate combinatorial condition which, if satisfied, insures that 
there exist infinitely many cyclic covers of $M(\A)$ whose 
first homology group has integer torsion.

\begin{theorem}
\label{thm:del multinet}
Suppose $\A$ admits a pointed multinet structure, with multiplicity 
vector $m$, and distinguished hyperplane $H$.  
Set $\A'=\A\setminus \{H\}$, and let $p$ be a prime dividing $m_H$. 
Then, for any sufficiently large integer $q$ relatively prime to $p$, 
there exists a regular, $q$-fold cyclic cover $Y\to U(\A')$ 
such that $H_1(Y,\Z)$ has $p$-torsion.
\end{theorem}

\begin{proof}
By Proposition \ref{prop:del}, the complement $M(\A')$ --- and 
hence, its projectivization, $U(\A')$ --- admits a small pencil.  
The desired conclusion follows from Theorem \ref{thm:flat pencil}. 
\end{proof}

Of course, additional information about the characteristic varieties of
the arrangement $\A'$ leads to sharper conclusions, as we see in the 
next example.

\begin{example}
\label{ex:deleted B3 covers}
We continue Example~\ref{ex:deleted B3} to see that 
Theorem~\ref{thm:del multinet} is constructive.  
Here, $\A'$ is the deleted ${\rm B}_3$ arrangement, and 
$p=2$.  For any odd integer $r>1$, let 
$\chi_r\colon\pi_1(U(\A'))\to \Z_r$ be the 
homomorphism given on generators by the vector 
$(2,-2,0,0,-1,-1,1,1)\in (\Z_r)^8$.  Let $X_r$ denote 
the corresponding regular $r$-fold cover of $U(\A')$.

We know from \cite{Su02} that $\V^1(\A')$ contains no isolated 
points.  Our arguments combine to show that, for all 
odd $r$, the homomorphism $\chi=\chi_r$ satisfies the 
hypotheses of Corollary \ref{cor:torschi}.  That is, the 
cyclic group $\im(\widehat{\chi}_\C)$ intersects $\V^1(\A')$ only
at the trivial representation $\set{\bo}$, while for $\k$ of characteristic
$2$, the image of $\widehat{\chi}_\k$ is contained in the 
$1$-parameter subgroup 
\[
T_\k=\set{(t^2,t^{-2},1,1,t^{-1},t^{-1},t,t)\mid t\in\k^*}\subseteq
\V^1(\A',\k).
\] 

It follows that $\dim_\C H_1(X_r,\C)=7$, while
$\dim_\k H_1(X_r,\k)\geq 7+(r-1)$.  Hence, $H_1(X_r,\Z)$ 
has $2$-torsion of rank at least $r-1$.
By Corollary~\ref{cor:alg mono}, the characteristic polynomials of
the monodromy do depend on the field: while $\Delta_1^{\C}(t)=(t-1)^7$, 
we see $\Delta_1^{\k}(t)$ is divisible by $(t-1)^6(t^r-1)$.
\end{example}

\subsection{The Milnor fibration of a multiarrangement}
\label{subsec:mf}

We now turn our consideration to a family of finite, cyclic covers of projective
hyperplane complements that arise geometrically as polynomial level sets.

As before, let $\A$ be an arrangement in $\C^{\ell}$.  For each 
hyperplane $H\in \A$, pick a linear form $f_H$ with $\ker(f_H)=H$. 
Let $\N=\Z_{>0}$, and 
let $m\in \N^{\A}$ be a choice of multiplicities on the arrangement. 
We will assume from now on that $\gcd(m_H\colon H\in \A)=1$.
Given these data, consider the polynomial 
\begin{equation}
\label{eq:qmulti}
Q(\A,m)=\prod_{H\in \A} f_H^{m_H}.
\end{equation}

The polynomial map $Q(\A,m)\colon \C^{\ell} \to \C$ restricts 
to a map $f=f_{\A,m}\colon M(\A) \to \C^{*}$.  As shown by 
J.~Milnor \cite{Mi} in a much more general context, $f$ 
is the projection map of a smooth, locally trivial bundle, 
known as the {\em Milnor fibration}\/ of the multiarrangement. 
The typical fiber of this fibration, $F(\A,m)=f^{-1}(1)$, is called the 
{\em Milnor fiber}; it is a connected, smooth affine variety, 
having the homotopy type of an $(\ell-1)$-dimensional, finite 
CW-complex. 

Set $N=\sum_{H\in \A} m_H$, and let $\xi = \exp(2\pi i/N)$.
The geometric monodromy of the Milnor fibration is 
the map  $h\colon F(\A,m)\to F(\A,m)$ given by 
$(z_1,\dots,z_\ell) \mapsto (\xi z_1,\dots,\xi z_\ell)$. 
The map $h$ generates a cyclic group
of order $N$; this group acts freely on $F(\A,m)$, with quotient 
the projectivized complement, $U(\A)$. We thus 
have a regular, $N$-fold cover, $F(\A,m) \to U(\A)$.
As noted in \cite{CS95, CDS03}, this cover can be 
described as follows.

\begin{proposition}
\label{prop:mf} 
The Milnor fiber cover $F(\A,m) \to U(\A)$ is classified 
by the homomorphism 
$\delta=\delta_{\A,m}\colon \pi_1(U)\surj \Z_N$ 
which sends $x_H\mapsto m_H\bmod N$, 
for each $H\in \A$.
\end{proposition}

\begin{proof} (Sketch)
The fibration $f\colon M(\A)\to \C^*$ induces a surjective  
homomorphism on fundamental groups, 
$\gamma=\gamma_{\A,m}\colon \pi_1(M(\A)) \surj \Z$, 
given on generators by $\delta(x_H)=m_H$.  
In view of Lemma \ref{lem:H1ofA}, 
the map $\delta$ induces a well-defined homomorphism 
$\delta\colon \pi_1(U(\A)) \surj \Z_N$, 
given by $\delta(x_H)=m_H \bmod N$. 
It is now readily seen that this homomorphism is 
a classifying map for the cover $F(\A,m) \to U(\A)$.  
\end{proof}

For full details on the proof of this proposition, we refer 
to \cite{Su13b}. 
Note that $F(\A,m)$ is the homotopy fiber of the infinite 
cyclic cover of $M(\A)$ determined by the homomorphism 
$\gamma_{\A,m}$.  We summarize this discussion in the 
following commuting diagram, with exact rows and columns:
\begin{equation}
\label{eq:delta cd}
\xymatrixcolsep{28pt}
\xymatrix{
& \Z \ar@{=}[r] \ar@{^{(}->}[d] & \Z \ar@{^{(}->}^(.45){\times N}[d] \\
\pi_1(F(\A,m))\ar@{^{(}->}[r] \ar@{=}[d] & \pi_1(M(\A))  \ar@{->>}[d] 
\ar@{->>}^(.6){\gamma_{\A,m}}[r] & \Z \ar@{->>}[d]\\
\pi_1(F(\A,m))\ar@{^{(}->}[r] & \pi_1(U(\A)) 
\ar@{->>}^(.65){\delta_{\A,m}}[r] & \Z_N 
}
\end{equation}

\subsection{The characteristic polynomial of the algebraic monodromy}
\label{subsec:char poly}

Under the above identification of the Milnor fiber $F(\A,m)$ as 
the regular $\Z_N$-cover $U(\A)^{\delta_{\A,m}}$, the geometric 
monodromy $h\colon F(\A,m)\to F(\A,m)$ of the Milnor fibration 
coincides with the monodromy $h=h_{\alpha}$ of the 
cover, for the choice of generator $\alpha=1\in \Z_N$.

As usual, let $\k$ be an algebraically closed field of 
characteristic not diving $N$. 
For each $q\ge 0$, consider the algebraic 
monodromy, $h_*\colon H_q(F(\A,m),\k) \to H_q(F(\A,m),\k)$, 
and let $\Delta^{\k}_q(t)=\det (t \cdot \id - h_*)$ be its 
characteristic polynomial.  It is readily seen that  
$\Delta^{\k}_q(t)$ coincides with the $q$-th Alexander polynomial 
of the infinite cyclic cover of $M(\A)$ determined by $\gamma_{\A,m}$: 
that is to say, $\Delta^{\k}_q(t)$ is the monic generator of the annihilator 
of $H_q(F(\A,m),\k)$, viewed as a $\k[\Z]$-module.

Using Theorem \ref{thm:eko} and Corollary \ref{cor:alg mono}, 
we obtain a concrete description of the homology groups of the 
Milnor fiber, and the eigenspace decomposition of the algebraic 
monodromy, solely in terms of the characteristic varieties of 
the arrangement.

As usual, fix an ordering on the hyperplanes of $\A$, 
and identify $\Hom (\pi_1(M(\A)),\k^*)$ with $(\k^*)^{\abs{\A}}$.  
For an element $\zeta\in \k^*$, set 
$\zeta^m:=(\zeta^{m_H})_{H\in \A}\in (\k^*)^{\abs{\A}}$.  

\begin{corollary}
\label{cor:hmf}
The homology groups of the Milnor fiber $F(\A,m)$, 
with coefficients in a field $\k$ of characteristic 
not dividing $N=\sum_{H\in \A} m_H$ are given by
\begin{equation}
\label{eq:eko bis}
\dim_\k H_q(F(\A,m),\k)=\sum_{d\geq 1}\abs{\V^q_d (U,\k)\cap 
\im(\widehat{\delta_{\A,m}})}.
\end{equation}

Moreover, the characteristic polynomial of the algebraic monodromy 
is given by
\begin{equation}
\label{eq:charpoly bis}
\Delta^{\k}_q(t) =  \prod_{d\ge 1} 
\prod_{\substack{\zeta\in \k^* \colon \zeta^N=1, \\
\zeta^m \in  \V^q_d(U,\k)}}  (t-\zeta).
\end{equation}
\end{corollary}

Particularly interesting is the case when all the multiplicities 
$m_H$ are equal to $1$.  In this situation, the polynomial 
$Q(\A)=Q(\A,m)$ is the usual defining polynomial for the 
arrangement, and has degree $n=\abs{\A}$.  Moreover, 
$F(\A)=F(\A,m)$ is the usual Milnor fiber of $\A$, corresponding 
to the ``diagonal" homomorphism, 
$\delta=\delta_{\A}\colon \pi_1(U(\A))\surj \Z_n$, 
which takes each meridian $x_H$ to $1$.

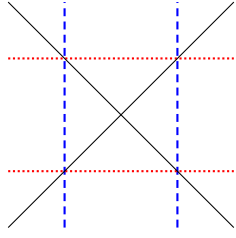
\begin{figure}
\begin{tikzpicture}[scale=0.75]
\draw[style=thick,densely dashed,color=blue] (-1,-2) -- (-1,2);
\draw[style=thick,densely dashed,color=blue] (1,-2) -- (1,2);
\draw[style=thick,densely dotted,color=red] (-2,-1) -- (2,-1);
\draw[style=thick,densely dotted,color=red] (-2,1) -- (2,1);
\draw (-2,-2) -- (2,2);
\draw (-2,2) -- (2,-2);
\end{tikzpicture}
\caption{A $(3,2)$-net on the ${\rm A}_3$ arrangement}
\label{fig:A3}
\end{figure}

\begin{example}
\label{ex:braid arr}
Let $\A$ be the braid arrangement in $\C^3$, defined by the 
polynomial $Q=(x^2-y^2)(x^2-z^2)(y^2-z^2)$.  The variety 
$\V^1(\A)\subset (\C^*)^6$ has $4$ local components 
of dimension $2$, corresponding to $4$ triple points in a generic 
section.  Additionally, the rational map 
$f\colon \P^2 \dashrightarrow \P^1$ 
given by $f([x\co y\co z])=[x^2-y^2\co x^2-z^2]$ restricts to 
a holomorphic fibration $f\colon M(\A)\to \Sigma_{0,3}$, 
where  $\Sigma_{0,3}=\P^1 \setminus \{ [1\co 0], [0\co 1], [1\co 1] \}$. 
This yields an essential, $2$-dimensional component in $\V^1(\A)$, 
corresponding to the multinet depicted in Figure \ref{fig:A3}, in 
which all multiplicities are equal to $1$.

Let $F(\A)$ be the Milnor fiber, defined by the diagonal character 
$\delta\colon \pi_1(U(\A))\to \Z_6$, and let $\iota\colon\Z_6\to\C^*$ 
be an injective homomorphism.  It is readily seen that 
$(\iota\circ\delta)^2\in \V^1_1(U(\A))$, yet 
$\iota\circ\delta\notin \V^1_1(U(\A))$.  As a consequence, 
we recover the well-known fact that 
$b_1(F(\A))=7$ and $\Delta^{\C}_1(t)=(t-1)^5(t^2+t+1)$.  

It turns out that the hypothesis of Proposition \ref{prop:no torsion} are 
satisfied for $X=U(\A)$ and $q=1$.  Thus, $\Delta^{\k}_1(t)=\Delta^{\C}_1(t)$, 
and 
\[
\Delta^{\k}_{U(\A),\delta}(\mathbf{u},x)=
u_1+(5u_1+2u_3)x+(6u_1+2u_2+6u_3+4u_6)x^2,  
\]
for all fields $\k$ of characteristic different from $2$ or $3$. 
Moreover, a Fox calculus computation shows that, in fact, 
$H_1(F(\A),\Z)=\Z^7$. 
\end{example}

\subsection{Domination by Milnor fibers}
\label{subsec:dominate}
Not every regular, cyclic cover of a projective arrangement complement
arises through the Milnor fiber construction. Nevertheless, it is not difficult
to say which ones do.

\begin{lemma}
\label{lem:F covers}
Let $(\A,m)$ be a multiarrangement, and let $U^{\chi} \to U$ 
be a regular, $\Z_N$-fold cover of $U=\P M(\A)$, where 
$N=\sum_{H\in \A} m_H$. 
Such a cover is equivalent to the Milnor fiber $F(\A,m)$ if 
and only if, for some $k$ relatively prime to $N$, the least 
strictly positive integers $\mu_H$ satisfying $k\cdot \chi(x_H)=\mu_H 
\bmod N$ have the property that $\sum_{H\in\A} \mu_H=N$, in 
which case $\mu=m$.
\end{lemma}

\begin{proof}
Recall that $F(\A,m)$ is the regular $\Z_N$-cover of $U$ given by 
the homomorphism $\delta\colon \pi_1(U)\surj \Z_N$, 
$x_H \mapsto m_H \bmod N$, and it has the property that 
$\sum_{H\in\A} m_H=N$.

Let $U^{\chi} \to U$ be a cover given by a homomorphism 
$\chi\colon \pi_1(U)\surj \Z_N$.  Note that any of the homomorphisms
$\set{k\cdot\chi\colon (k,N)=1}$ determine the same cover.
By Lemma~\ref{lem:H1ofA}, we have $\sum_{H\in\A}\chi(x_H)=0$ in $\Z_N$.
So, given $k$ relatively prime to $N$, if $\set{m_H}_{H\in \A}$ are positive 
integers for which $m_H=k\chi(x_H)\bmod N$, then 
$N\mid\sum_{H\in\A} m_H$.  Thus, we see that 
$\lambda:=k\cdot\chi$ determines a cover which is a 
Milnor fiber if and only if $N=\sum_{H\in\A} m_H$,  
for some choice of $k$.  This ends the proof.
\end{proof}

The Milnor fibers associated with a given arrangement dominate 
all other cyclic covers, in the following sense.

\begin{proposition}
\label{prop:Milnor dominates}
Let $\A$ be an arrangement, and let $U^\chi$ be a regular, 
$r$-fold cyclic cover of $U=\P M(\A)$, given by a surjective 
homomorphism $\chi\colon\pi_1(U)\to\Z_r$.  Then there exist 
infinitely many multiplicity vectors $m$ for which the covering 
projection $F(\A,m)\to U$ factors through $U^\chi$:
\[
\xymatrix{
F(\A,m)\ar[dr]\ar@{.>}[r] & U^\chi\ar[d]\\
&U.
}
\]
Moreover, for any prime $p$ not dividing $r$, we may choose 
$m$ so that the degree of $Q(\A,m)$ is not divisible by $p$.
\end{proposition}

\begin{proof}
Let $n=\abs{\A}$, and let $c\in\Z_r^n$ be given by 
$c_H=\chi(x_H)$, for each $H\in\A$. Also, 
let $S=\set{kc\colon (k,r)=1}$, let $\widetilde{S}\subseteq \Z^n$ 
denote the preimage of the set $S$ under the quotient map 
$\pr\colon \Z^n\twoheadrightarrow \Z_r^n$, and let 
$\widetilde{S}_{>0}=\widetilde{S}
\cap \N^n$ be its intersection with the positive orthant.

Consider any lattice point $m\in
\widetilde{S}_{>0}$, and let $N=\sum_{H\in\A}m_H$.  Let 
$\delta(x_H)=m_H\bmod N$ for all $H\in\A$.  By Proposition \ref{prop:mf}
the cover $U^\delta$ is equivalent to $F(\A,m)$.  Since $\pr(m)=c$,
by construction, the projection $U^\delta\to U$ factors through
$U^\chi \to U$. In fact, by Lemma~\ref{lem:F covers}, the set
$\widetilde{S}_{>0}$ parameterizes all possible multiplicities $m$ 
for which $F(\A,m)\to U$ factors through the given cover $U^\chi \to U$.  
By covering space theory, the lift $F(\A,m)\to U^{\chi}$ is also 
a regular cover.

To prove the last assertion, recall that the degree of the 
polynomial $Q(\A,m)$ equals $N:=\sum_{H\in\A}m_H$.
If, more generally, $p$ is any integer relatively prime 
to $r$ and $p>1$, it follows from the Chinese Remainder 
Theorem that the set $\set{m\in\widetilde{S}\colon p\nmid N}$
is nonempty and closed under translation by the lattice $(pr\Z)^n$.  In
particular, its intersection with the positive orthant must contain
infinitely many points.
\end{proof}

In the case when all the multiplicities are $1$, we can be more precise. 
As usual, let $\A$ be an arrangement, with defining polynomial 
$Q=\prod_{H\in\A}f_H$; let $F=Q^{-1}(1)$ be the Milnor fiber,   
and let $U=\P M(\A)$. From Proposition \ref{prop:mf}, we know 
that the cyclic cover $F\to U$ is classified by the class 
$\delta_{\A}\in H^1(U,\Z_n)$, 
where $n$ is the number of hyperplanes of $\A$.  
Proposition~\ref{prop:Milnor dominates} then 
simplifies to the following.

\begin{corollary}
\label{cor:Milnor dominates2}
Let $\chi\colon\pi_1(U)\surj\Z_r$ be a surjective homomorphism and
let $U^\chi \to U$ be the corresponding cyclic cover.  The Milnor fiber 
cover $F\to U$ factors through $U^\chi$ if and only if $r\mid n$, and
$\chi\equiv k\delta_\A \bmod r$, for some integer $k$ relatively 
prime to $r$.
\end{corollary}

\subsection{Torsion in the Milnor fiber of a multiarrangement}
\label{subsec:mf multiarr}
We are now ready to complete the proof of 
Theorem \ref{thm:intro multi tors} from the Introduction. 

\begin{theorem}
\label{thm:milnor multi}
Let $\A$ be hyperplane arrangement which admits a pointed 
multinet structure, with multiplicity vector $m$, and distinguished 
hyperplane $H$.  Set $\A' =\A\setminus \{H\}$, 
and let $p$ be a prime dividing $m_H$. 
Then, there is a choice of multiplicities $m'$ on $\A'$ 
such that the Milnor fiber of $(\A',m')$ has $p$-torsion 
in first homology.
\end{theorem}

\begin{proof}
By Theorem~\ref{thm:del multinet}, there exists an integer $q$, 
not divisible by $p$, and a $q$-fold cyclic cover $Y\to U(\A')$ 
for which $H_1(Y,\Z)$ has non-trivial $p$-torsion.

By Proposition~\ref{prop:Milnor dominates}, there exists a choice 
of multiplicities $m'\in \N^n$ for which the Milnor fiber $F(\A',m')$ 
factors through $Y$, and $p$ does not divide $N' :=\sum_{H\in\A'}m'_H$.  
By Lemma~\ref{lem:transfer}, there is also $p$-torsion in $H_1(F(\A',m'),\Z)$.
\end{proof}

\begin{remark}
\label{rem:avoid2}
Since the choices of multiplicities $m$ given by 
Proposition~\ref{prop:Milnor dominates} are closed under translation by
elements of the semigroup $(pr\N)^n$, it should be clear that we may
also choose the multiplicities $m'$ in Theorem \ref{thm:milnor multi} 
to have $m'_{K}\neq2$ for each $K\in \A'$, an additional technical 
condition we will require in \S\ref{sect:polar}.
\end{remark}

We conclude this section with a few examples to
illustrate the range of applicability of Theorem \ref{thm:milnor multi}.

\begin{example}
\label{ex:mf deleted B3}
Let $\A'$ be the deletion of the $\operatorname{B}_3$ arrangement 
from Examples \ref{ex:deleted B3} and \ref{ex:deleted B3 covers}.
Proceeding as above, let $r=3$ and consider the cover given by 
$\chi_3=(2,1,0,0,2,2,1,1)\in(\Z_3)^8$.  Following the construction 
in the proof of Proposition~\ref{prop:Milnor dominates}, we may 
choose $m'=(2,1,3,3,2,2,1,1)$ as our positive integer lift of $\chi_3$.
Here $N'=15$, and $F(\A',m')$ factors through the $3$-fold cover 
indexed by $\chi_3$.  

By Theorem~\ref{thm:milnor multi}, we see that there is $2$-torsion 
in the first homology of $F(\A',m')$.  It is not hard to see that $15$ is 
the least possible order of a Milnor fiber cover factoring through 
the given $3$-fold cover.  Explicit calculation shows that, in fact, 
$H_1(F(\A',m'),\Z)=\Z^7\oplus \Z_2\oplus \Z_2$ and 
$\Delta_1^{\overline{\F_2}}(t)=(t-1)^7(t^2+t+1)$. This, then, 
recovers and generalizes a computation from \cite{CDS03}.

As noted in Remark~\ref{rem:avoid2}, we could also choose 
$m'=(8,1,3,3,5,5,1,1)$ as a lift of $\chi_3$ to avoid the multiplicity 
$2$; the group $H_1(F(\A',m'),\Z)$ stays the same. 
\end{example}

\begin{example}
\label{ex:monomial}
Now let $p$ be an odd prime, and let $\A(p,1,3)$ be the reflection
arrangement of the full monomial complex reflection group.  Let 
$Q_1=x^p(y^p-z^p)$, $Q_2=y^p(x^p-z^p)$ and $Q_3=z^p(x^p-z^p)$: 
then $Q_1Q_2Q_3$ is a defining polynomial for a multiarrangement 
$(\A(p,1,3),m)$.  As observed in \cite{FY}, this factorization gives 
$(\A(p,1,3),m)$ the structure of a multinet.  

Choosing $H=\set{x=0}$ as the distinguished hyperplane, it is 
straightforward to check that $(\A(p,1,3),m,H)$ is a pointed multinet.  
Let $\A'_p=\A(p,1,3)\setminus \set{H}$, and order the hyperplanes 
compatibly with an ordered factorization of
 $yz(x^p-y^p)(x^p-z^p)(y^p-z^p)$.  Then
$\rho T_\k\subseteq \V^1(\A'_p,\k)$, where
\begin{align*}
T_\k&=\set{(t^p,t^{-p},\overbrace{t^{-1},\ldots,t^{-1}}^p,
\overbrace{\vphantom{t^1}t,\ldots,t}^p,
\overbrace{\vphantom{t^1}1,\ldots,1}^p)\colon t\in\k^*}\text{~and}\\
\rho &=(1\,,1\,,\zeta^{-1},\ldots,\zeta^{-1},1,\ldots,1,\zeta,\ldots,\zeta),
\end{align*}
where $\zeta$ is a primitive $p$th root of unity if $p\nmid\ch(\k)$, and
$\zeta=1$ otherwise.  Let $q$ be any prime with $q>p$, and let
\begin{equation}
\label{eq:monomial m}
m'=(p,q-p,\overbrace{q-1,\ldots,q-1}^p,\overbrace{1,\ldots,1}^p,
\overbrace{\vphantom{1}q,\ldots,q}^p).
\end{equation}
Then $N'=(2p+1)q$ is relatively prime to $p$, so the Milnor fiber 
cover $F(\A'_p,m')$ factors through a $q$-fold cyclic cover, and 
$H_1(F(\A'_p,m'),\Z)$ has nonzero $p$-torsion.  This generalizes 
Example~\ref{ex:mf deleted B3}, and again recovers a computation 
from \cite{CDS03}.
\end{example}

\begin{remark}
\label{rem:more multinets}
The only other multiarrangements known to us for which some multiplicity 
is greater than $1$ form another infinite family, introduced in 
\cite[Example~4.10]{FY}.  Again, the corresponding multinets 
are pointed, and so they have deletions with torsion in the first homology 
of suitable cyclic covers.  It would be interesting to know of other examples.
\end{remark}

\section{The parallel connection operad} 
\label{sect:operad}

We now turn to parallel connections of matroids and 
arrangements. To help organize the constructions that 
follow, we note here that parallel connection can be 
viewed in terms of an operad. 

\subsection{The parallel connection of pointed matroids}
\label{subsec:par mat}
We start with a definition.  We refer to the book of 
Markl, Shnider and Stasheff~\cite{MSS02}, as well as
the survey of Chapoton~\cite{Ch07} for its combinatorial 
point of view, and we also refer to \cite[\S 7.1]{Ox92} for
a general discussion of parallel connections.

\begin{definition}
\label{def:matcat}
Let $\Mat$ denote the category whose objects are matroids and 
morphisms are simply isomorphisms.  A {\em pointed matroid}\/ 
is a pair $(E,e)$, where $E$ is a matroid and $e$ is a point in 
$E$ which is not a loop.  We will call $e$ the {\em basepoint}\/ 
of $(E,e)$.

Let $\PMat$ denote the category of pointed matroids, whose 
morphisms are basepoint-preserving isomorphisms.  Let $\PMat(n)$ 
be the set of pointed matroids with $n$ points, for $n\geq1$.  
We will denote the pointed matroid with only one point 
by $\unit=(\set{1},1)$.
\end{definition}

If $E_1$ and $E_2$ are matroids containing points $e_1$, $e_2$, 
respectively, we will denote the parallel connection identifying 
$e_1$ and $e_2$ by $E_1\PC{e_1}{e_2} E_2$.

\subsection{An operad structure}
\label{subsec:oper}
While it is natural to connect pointed matroids along their basepoints
(as in \cite{Ox92}), we shall make use of a slightly less restrictive
construction.
\begin{definition}
\label{def:pc mat}
If $(E_i,e_i)$ are pointed matroids for $i=1,2$ and $x\in E_1$, let
\[
(E_1,e_1)\plugin_x(E_2,e_2)=(E_1\PC{x}{e_2}E_2, e_1).
\]
\end{definition}

\def\dx{0.5}
\def\dy{0.71}
\def\tilt{-35.15} % angle through (1)--(5)
\newcommand{\drawE}[1]
{
\begin{scope}[rotate around={#1:(\dx,-\dy)}]
\node[root] (1) at (0,0) [label=180+#1:$e_1$] {};
\node[plain] (2) at (-\dx,-\dy) {};
\node[plain] (4) at (-2*\dx,-2*\dy) {};
\node[plain] (5) at (\dx,-\dy) [label=right:$x$] {};  
\node[plain] (3) at (2*\dx,-2*\dy) {};
\node[plain] (6) at (0,-4/3*\dy) {};
\draw (1) -- (2) -- (4) -- (6) -- (5);
\draw (1) -- (5) -- (3);
\draw (2) -- (6) -- (3);
\end{scope}
}
\begin{figure} % [h]
\begin{tikzpicture}[scale=0.75,baseline=(current bounding box.center),
plain/.style={circle,draw,inner sep=1.5pt,fill=white},
root/.style={circle,draw,inner sep=1.5pt,fill=black}]
\drawE{0}
\def\shift{2.8}
\node (k0) at (\shift,-\dy) [root,label=above:$e_2$] {};
\node at (0.5*\dx+0.5*\shift,-\dy) {$\plugin_x$};
\draw (k0) -- (\shift+1,-\dy) node[plain] {} -- (\shift+2,-\dy) node[plain] {};
\end{tikzpicture}
$\qquad\mapsto\qquad$
\begin{tikzpicture}[scale=0.75,baseline=(current bounding box.center),
plain/.style={circle,draw,inner sep=1.5pt,fill=white},
root/.style={circle,draw,inner sep=1.5pt,fill=black}]
\drawE{\tilt}
\draw (5) -- (\dy+\dx,0) node[plain] {} -- (2*\dy+\dx,\dy) node[plain] {};
\draw[style=help lines,dashed] (\dx,-2.9) -- (\dx,1.2);
\draw[style=help lines,dashed] (\dx+0.4,-2.5) -- (-2,-2.5);
\draw[style=help lines,dashed] (0.1,-2.9) -- (3.0,0);
\end{tikzpicture}
\caption{The parallel connection pseudo-operad}
\label{fig:pc_example}
\end{figure}
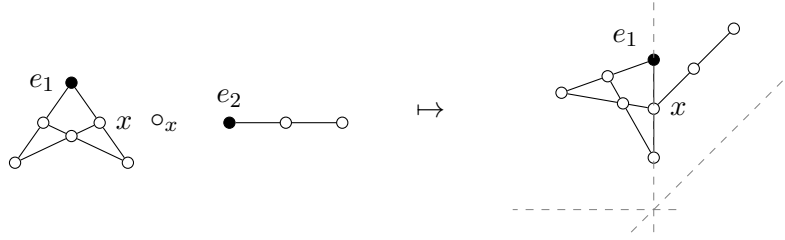

The construction is illustrated in Figure \ref{fig:pc_example}.

We have obvious natural identifications 
$\unit\plugin_1(E,e)\cong (E,e)\plugin_x\unit\cong(E,e)$ 
for any pointed matroid $(E,e)$ and $x\in E$.  Since the parallel 
connection operation is associative, we observe:

\begin{proposition}
\label{prop:operad}
The operations $\plugin_x\colon \PMat(m)\sqcup\PMat(n)\to\PMat(m+n-1)$ 
define a symmetric operad of sets with unit $\unit$.
\end{proposition}

\subsection{A linear map}
\label{subsec:hmap}

If $(E,e)$ is a pointed matroid, let $\Ab(E,e)$ denote the free abelian 
group on the set of points $E$.  If $(E_i,e_i)$ are pointed matroids
for $i=1,2$ and $x\in E_1$, define a homomorphism
\begin{equation}
\label{eq:hmap}
\xymatrix{\plugin_x\colon
\Ab(E_1,e_1)\oplus \Ab(E_2,e_2)\ar[r]&\Ab\big((E_1,e_1)\plugin_x(E_2,e_2)\big)
}
\end{equation}
by its action on the basis:
\begin{equation}\label{eq:hmap_formula}
e\mapsto \begin{cases} e & \text{if $e\notin \set{x,e_2}$;}\\
 \sum_{f\in E_2} f& \text{if $e=x$;}\\
 \sum_{f\in E_1} f& \text{if $e=e_2$.}
\end{cases}
\end{equation}

Let $\FF$ denote the category of free abelian groups 
on finite, pointed sets.  Direct sum gives $\FF$ the structure of a 
symmetric, monoidal category.  By inspecting the associativity 
axioms we see:

\begin{proposition}
\label{prop:operad2}
The homomorphisms $\plugin_x$ from \eqref{eq:hmap} 
define a symmetric operad on $\FF$, with unit $\id\colon\FF\to\FF$.
\end{proposition}

We record the following easy observation for later use. Let $\PAb(E,e)$ 
denote the quotient of $\Ab(E,e)$ by the sum of the points of $E$.

\begin{proposition}
\label{prop:proj operad2}
The homomorphism \eqref{eq:hmap} descends to quotients, 
yielding an isomorphism 
\begin{equation}
\label{eq:Phmap}
\xymatrix{\plugin_x\colon
\PAb(E_1,e_1)\oplus \PAb(E_2,e_2)\ar[r]& 
\PAb\big((E_1,e_1)\plugin_x(E_2,e_2)\big)
}.
\end{equation}
\end{proposition}
For an element $v\in\Ab(E,e)$, let $\set{v_f}_{f\in E}$ denote its
coordinates.  For $i=1,2$ and $\overline{v_i}\in \PAb(E_i,e_i)$, 
it is convenient to choose representatives with coordinates 
$(v_1)_x=0$ and $(v_2)_{e_2}=0$.  Then \eqref{eq:hmap_formula} 
simply says
\begin{equation}
\label{eq:Phmap_coords}
(\overline{v_1}\plugin_x\overline{v_2})_e=
\begin{cases}
(v_1)_e & \text{if $e\in E_1$;}\\
(v_2)_e & \text{if $e\in E_2$.}
\end{cases}
\end{equation}

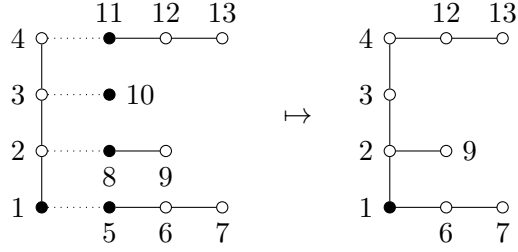
\begin{figure}
\newcommand{\drawBackbone}
{
\node[root] (h0) at (0,0) [label=left:1] {};
\node[plain] (h1) at (0,1) [label=left:2] {};
\node[plain] (h2) at (0,2) [label=left:3] {};
\node[plain] (h3) at (0,3) [label=left:4] {};
\draw (h0) -- (h1) -- (h2) -- (h3);
}
\[
\begin{tikzpicture}[scale=0.75,baseline=(current bounding box.center),
plain/.style={circle,draw,inner sep=1.5pt,fill=white},
root/.style={circle,draw,inner sep=1.5pt,fill=black}]
\def\shift{1.2}
\drawBackbone
\node (k0) at (\shift,0) [root,label=below:5] {};
\node (k1) at (\shift,1) [root,label=below:8] {};
\node (k2) at (\shift,2) [root,label=right:10] {};
\node (k3) at (\shift,3) [root,label=above:11] {};

\draw[dotted] (h0) -- (k0);
\draw[dotted] (h1) -- (k1);
\draw[dotted] (h2) -- (k2);
\draw[dotted] (h3) -- (k3);
\draw (k0) -- (\shift+1,0) node [plain,label=below:6] {} -- 
              (\shift+2,0) node [plain,label=below:7] {};
\draw (k1) -- (\shift+1,1) node [plain,label=below:9] {};
\draw (k3) -- (\shift+1,3) node [plain,label=above:12] {} -- 
              (\shift+2,3) node [plain,label=above:13] {};
\end{tikzpicture}
\quad\mapsto\quad
\begin{tikzpicture}[scale=0.75,baseline=(current bounding box.center),
plain/.style={circle,draw,inner sep=1.5pt,fill=white},
root/.style={circle,draw,inner sep=1.5pt,fill=black}]
\drawBackbone
\draw (h0) -- (1,0) node [plain,label=below:6] {} -- (2,0) 
                    node [plain,label=below:7] {} ;
\draw (h1) -- (1,1) node [plain,label=right:9] {};
\draw (h3) -- (1,3) node [plain,label=above:12] {} -- (2,3) 
                    node [plain,label=above:13] {} ;
\end{tikzpicture}
\]
\caption{Iterated parallel connections of pointed matroids}
\label{fig:pc}
\end{figure}

\begin{proposition}
\label{prop:proj operad2 dual}
The isomorphism
\begin{equation}
\label{eq:Phmap bis}
\xymatrix{\plugin_x^*\colon
\PAb\big((E_1,e_1)\plugin_x(E_2,e_2)\big)^*\ar[r] & 
\PAb(E_1,e_1)^*\oplus \PAb(E_2,e_2)^*}
\end{equation}
is given for $w\in \PAb\big((E_1,e_1)\plugin_x(E_2,e_2)\big)^*$ by
\begin{equation}\label{eq:pluginstar}
(\plugin_x^*(w))_e =\begin{cases}
w_e & \text{for $e\neq x,e_2$;}\\
w_e-\sum_{f\in E_2}w_f & \text{for $e=e_2$;}\\
\sum_{f\in E_2}w_f & \text{for $e=x$.}
\end{cases}
\end{equation}
\end{proposition}

\begin{proof}
We note that $\PAb(E,e)^*$ is the subgroup of
$\Ab(E,e)^*$ of functions vanishing on the sum of the points, so we
dualize \eqref{eq:hmap_formula} and restrict.  The claim follows by
noting that $w_{e_2}-\sum_{f\in E_2}w_f=\sum_{f\in E_1}w_f$.
\end{proof}

\begin{example}
\label{ex:pc in h1}
To illustrate, consider the iterated parallel connection in 
Figure~\ref{fig:pc}.  The image of a function with value $1$ on all
but the root is shown, under the iterated application of $\plugin_x^*$.
In particular, let $E_1=\set{1,2,3,4}$ denote the uniform 
matroid drawn vertically on the right.  Then the projection of the
image to $\PAb(E_1,1)^*$ gives the function
$w=(-6,2,1,3)\in \PAb(E_1,1)^*$.  
\end{example}

\subsection{Parallel connections of arrangements}
\label{subsec:parallel}

Multiarrangements of hyperplanes are simply linear representations 
of matroids without loops.  In this section, we analyze the effects of parallel
connections on constructions such as the characteristic varieties of 
complex arrangement complements, introduced in \S\ref{subsec:hyp arr}.

Suppose $\A_1$ and $\A_2$ are central arrangements of hyperplanes 
in complex vector spaces $V_1$ and $V_2$, respectively, defined by 
polynomials $Q_1=\prod_{H\in\A_1}f_H$ and $Q_2=\prod_{H\in\A_2}g_H$. 
We assume here that $\A_1$ and $\A_2$ are simple arrangements.%

\begin{definition} 
\label{def:parcon}
For a choice of hyperplanes $H_1\in\A_1$ and $H_2\in\A_2$, 
the {\em parallel connection}, $\A_1\PC{H_1}{H_2}\A_2$, 
is the arrangement defined by the polynomial
\begin{equation}
\label{eq:parallel poly}
f_{H_1}\cdot \left(\prod_{H\in\A_1\setminus\{H_1\}}f_H\right) \cdot
\left(\prod_{H\in\A_2\setminus\{H_2\}}g_H\right),
\end{equation}
in the ring $\k[V_1^*]\otimes_\k \k[V_2^*]/(f_{H_1}-g_{H_2})$. 
\end{definition}

This is an arrangement in the codimension-$1$ subspace 
of $V\times W$ consisting of the pairs $(x,y)$ for which 
$f_{H_1}(x)=g_{H_2}(y)$.  Its underlying matroid is the 
parallel connection of those of $\A_1$ and $\A_2$.  
Figure \ref{fig:A3polar} illustrates this construction:  
we first realize the matroids from Figure \ref{fig:pc_example} 
by the braid arrangement and a pencil, perform 
a parallel connection, and then draw a generic planar slice.

\begin{figure}
\begin{tikzpicture}[scale=0.75]
\draw[color=blue](-1,-2) -- (-1,2);
\draw[style=thick, color=blue](1,-2) -- (1,2);
\draw[color=blue](-2,-1) -- (2,-1);
\draw[color=blue](-2,1) -- (2,1);
\draw[color=blue](-2,-2) -- (2,2);
\draw[color=blue](-2,2) -- (2,-2);
\node at (2.8,0) {$\pc$};
\draw[style=thick, color=red](5,-2) -- (5,2);
\draw[color=red](3.8,-1.5) -- (6.2,1.5);
\draw[color=red](3.8,1.5) -- (6.2,-1.5);
\node at (7.7,0) {$\mapsto$};
\draw[color=blue](10,-2) -- (10,2);
\draw[style=thick, color=blue!50!red](12,-2) -- (12,2);
\draw[color=blue](9,-1) -- (13,-1);
\draw[color=blue](9,1) -- (13,1);
\draw[color=blue](9,-2) -- (13,2);
\draw[color=blue](9,2) -- (13,-2);
\def\px{12}\def\py{0}
\foreach \theta in {-68,68} {
  \draw[color=red,rotate around={\theta:(\px,\py)}] 
       (\px-2.18,\py) -- (\px+2.18,\py);
}
\end{tikzpicture}
\caption{Parallel connection (in a generic planar slice view)}
\label{fig:A3polar}
\end{figure}
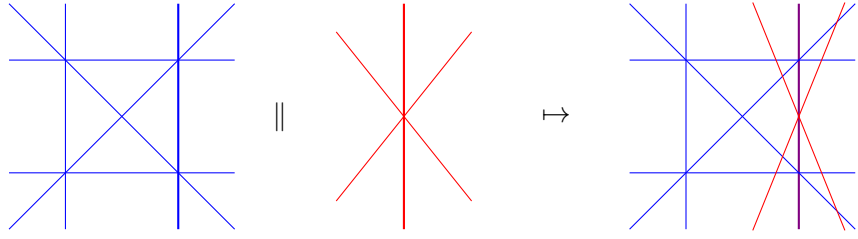

Let $M_1$ and $M_2$ denote the complements of $\A_1$ and $\A_2$, 
respectively.  We will denote the complement of $\A_1\PC{H_1}{H_2}\A_2$ 
by $M_1\PC{H_1}{H_2} M_2$.  

\begin{proposition}[cf.~\cite{EF, FP}]
\label{prop:falk}
Choose hyperplanes $H_1\in \A_1$ and $H_2\in \A_2$. 
The rational map
\[
\phi\colon \P(V)\times \P(W)\to \P(V\times W)
\]
given by $\phi([v],[w])=[(g_{H_2}(w)v,f_{H_1}(v)w)]$
restricts to a diffeomorphism of hyperplane complements,
\begin{equation}\label{eq:falkmap}
\phi\colon \P(M_1)\times\P(M_2)\to \P(M_1\PC{H_1}{H_2} M_2),
\end{equation}
whose inverse is given by $[(v,w)]\mapsto ([v],[w])$.
\end{proposition}

\begin{proof}
First, if $v\in M_1$ and $w\in M_2$, then clearly
\begin{equation}
\label{eq:fhgh}
f_{H_1}(g_{H_2}(w)v)=g_{H_2}(f_{H_1}(v)w).
\end{equation}
Since $f_{H_1}(v)w\neq0$, the restriction of $\phi$ in \eqref{eq:falkmap}
is well-defined and regular.  It is easy to check that $[(v,w)]\mapsto([v],[w])$ 
is also regular on the hyperplane complements, and it is inverse to $\phi$.
\end{proof}

We define {\em pointed arrangements} $(\A,H)$ just as in 
Definition~\ref{def:matcat}.  If $(\A_i,H_i)$ are pointed arrangements
for $i=1,2$ and $H\in \A_1$, as in Definition~\ref{def:pc mat} we let
\begin{equation}
\label{eq:plugin}
(\A_1,H_1)\plugin_H(\A_2,H_2)=(\A_1 \PC{H}{H_2}\A_2, H_1).
\end{equation}
In the same way, for complements of pointed arrangements,
let $\plugin_H\colon\P(M_1)\times\P(M_2)\to\P(M_1\PC{H}{H_2}M_2)$ 
denote the diffeomorphism \eqref{eq:falkmap}.

\subsection{Induced map in first homology} 
\label{subsec:h1map}

Let $(\A,H)$ be a pointed arrangement, which we may also regard as a 
pointed matroid.  Set $n=\abs{\A}$. Then, using the notation 
of \S\ref{subsec:hmap}, by Lemma~\ref{lem:H1ofA}, 
we have 
\begin{equation}
\label{eq:hah}
h(\A,H)\cong H_1(M(\A),\Z)\cong\Z^n\:\:\text{and}\:\: 
\PAb(\A,H)=H_1(\P M,\Z)\cong\Z^n \big\slash 
\Big(\sum_{K\in\A} x_K\Big) \big. .
\end{equation}

Our next result states that the effect of parallel connection on 
abelianized fundamental groups is described by the linear operad of
Proposition~\ref{prop:operad2}.

\begin{proposition}
\label{prop:H_1}  
Let $(\A_1,H_1)$ and $(\A_2,H_2)$ be pointed arrangements, 
and let $U_i=\P M(\A_i)$.  For any hyperplane $H\in \A_1$, 
under the identification $\PAb(\A,H)=H_1(U,\Z)$, 
the map $H_1(\plugin_H)\colon H_1(U_1,\Z)\oplus H_1(U_2,\Z)
\to H_1(U_1\plugin_H U_2, \Z)$ equals $\plugin_H$ from \eqref{eq:Phmap}.
\end{proposition}

\begin{proof}
Let $n_i=\abs{\A_i}$ and $T_i=(\C^*)^{n_i}/\C^*$ for $i=1,2$, 
and set $T=(\C^*)^{n_1+n_2-1}/\C^*$.  Consider the torus 
embeddings of \eqref{eq:torusembedding}:
\begin{equation}
\label{eq:torus square}
\vcenter{\vbox{\xymatrix{
U_1\times U_2 \ar[r]^{\plugin_H}\ar[d]  & U_1\plugin_H U_2\ar[d]\\
T_1\times T_2\ar[r]^{\plugin_H} & T
}}}
\end{equation}

By Lemma~\ref{lem:H1ofA}, the images of the vertical maps under $H_1(-,\Z)$
are isomorphisms, so it is sufficient to prove the claim for 
$T_1\times T_2\to T$.

For a torus $T$, we may identify $H_1(T,\Z)$ with the group of one-parameter
subgroups $X_*(T):=\Hom(\C^*,T)$.  For points $s\in T_1$ and $t\in T_2$,
order coordinates so that $H_i$ comes first in $\A_i$, and without loss
of generality, so that $H$ is last in $\A_1$. Then
the map $T_1\times T_2\to T$ is given in coordinates by
\[
((s_1,\ldots,s_{n_1}),(t_1,\ldots,t_{n_2}))\mapsto
(t_1s_1,\ldots,t_1s_{n_1-1},t_1s_{n_1},t_2s_{n_1},t_3s_{n-1},
\ldots,t_{n_2}s_{n-1}).
\]
This is a group homomorphism, and the identification $H_1(T,\Z)\cong
X_*(T)$ is functorial, so $H_1(\plugin_H)$ is the image of this
monomial map under $X_*(-)$.  This, in turn, is readily seen to 
be given by the formula \eqref{eq:Phmap}.
\end{proof}

\subsection{Characteristic varieties of parallel connections}
\label{subsec:cv pc}
We continue to let $U_i$ denote the projective complement of a 
(pointed) arrangement $(\A_i,H_i)$ for $i=1,2$.  Since 
$U_1\plugin_H U_2\cong U_1\times U_2$ by 
Proposition~\ref{prop:H_1}, we may express the characteristic
varieties of the parallel connection complement by means of 
the product formula~\eqref{eq:cvprod}.  That is, 
\begin{equation}
\label{eq:pc of cv}
\V^q(U_1\plugin_H U_2,\k)\cong \bigcup_{i=0}^q \V^i(U_1,\k)\times
\V^{q-i}(U_2,\k),
\end{equation}
under the restriction of the map $\plugin_H^*$ computed above.

Let $\Pl_1$ denote the arrangement of one point in $\C^1$, and 
for $n\geq 2$, let $\Pl_n$ denote a central arrangement of $n$ lines
$L_1,\ldots,L_n$ in $\C^2$, with basepoint $L_1$.  The latter 
arrangements are realizations of the rank-$2$ uniform matroids.  
Write $P_n=\P(M(\Pl_n))$, and note that $P_1=\set{\text{point}}$, 
whereas $P_n\cong \C\setminus \set{\text{$n-1$ points}}$, for $n\geq 2$. 

Parallel connections with the arrangements $\Pl_n$, for $n\geq3$, will
be building blocks in the polarization construction of \S\ref{sect:polar}.

\begin{proposition}
\label{prop:delta}
Let $\chi\colon\prod_{i=1}^n F_{m_i-1}\to \Z_N$ be a 
surjective homomorphism indexing a regular cover of 
$X=\prod_{i=1}^n P_{m_i}$, where $n>1$ and each $m_i\geq 3$.
If each restriction  $\chi_i\colon F_{m_i-1}\to \Z_N$ is also
surjective, and $\k$ is a field for which $\ch(\k)\nmid N$, then 
\begin{equation}
\label{eq:delta prod free}
\Delta_{X,\chi}^\k(\mathbf{u},x)=u_1\prod_{i=1}^n(1+(m_i-1)x) + 
\prod_{i=1}^n (m_i-2) \sum_{1<k\mid n} \phi(k) u_k x^n.
\end{equation}
\end{proposition}

\begin{proof}
Let  $\rho=(\rho_1,\dots,\rho_n)\in\im(\widehat{\chi}_\k)$ be a 
non-trivial character.  Since each $\chi_i$ is surjective, the 
characters $\rho_i$ are also non-trivial. 
By formula \eqref{eq:poin free}, we have that 
$\poin^{\k}(P_{m_i},\rho_i;x)=(m_i-2)x$. 
The conclusion then follows from the product
formula of Proposition~\ref{prop:cover of product}.
\end{proof}

\begin{example}
\label{ex:cdp}
A special class of arrangements obtained by iterated parallel 
connection was studied by Choudary, Dimca and Papadima 
in \cite{CDP} (see also \cite{Wi1} for recent work in this direction). 
Let $(n,m_1,\ldots,m_n)$ be positive integers, and consider 
the pointed arrangement
\begin{equation}
\label{eq:cdp}
\A=\Pl_n\plugin_{L_1} \Pl_{m_1}\plugin_{L_1} \Pl_{m_2}\plugin_{L_1}\cdots
\plugin_{L_1} \Pl_{m_n}.
\end{equation}

The intersection of $\A$ with a generic, affine plane consists of 
$m_1+\cdots+m_n$ lines.  The lines may be partitioned into $n$ classes,
for which the lines in each class are parallel, and two lines from different
classes intersect at a double point.  By applying formula \eqref{eq:pc of cv} 
and Proposition~\ref{prop:H_1} repeatedly, one may obtain a full description 
of the characteristic varieties $\V^q(\A,\k)$, for all $q\geq0$, thereby 
recovering a result from \cite{CDP}.
\end{example}

\section{The polarization of a multiarrangement}
\label{sect:polar}

In this section, we use a special type of iterated 
parallel connection, which we call polarization, 
to construct hyperplane arrangements for which 
the (usual) Milnor fiber has non-trivial torsion 
in homology.

\subsection{Iterated parallel connections}
\label{subsec:polar arr}
We saw in the previous section that the character groups 
of parallel connections are isomorphic to the character 
groups of the factor arrangements via a straightforward,  
yet nontrivial operadic isomorphism.  Moreover, this 
isomorphism restricts to the respective characteristic varieties.

Since the characteristic varieties of the line arrangements 
$\Pl_n$ are positive dimen\-sional for $n\geq3$ but are simple
in nature, we will iterate parallel connections of an arbitrary 
multiarrangement with a collection of line arrangements.  In doing
so, we can construct new arrangements whose characteristic varieties
vary with the choice of multiplicities.
By analogy with polarization for monomial ideals (see, for example,
\cite[\S3.2]{MiSt}), we will call this the polarization of the
multiarrangement. 

\begin{definition}
\label{def:polar}
The {\em polarization}\/ of  a multiarrangement $(\A,m)$,
denoted by $\A\pc m$, is the arrangement of 
$N:=\sum_{H\in\A}m_H$ hyperplanes 
given by
\begin{equation}
\label{eq:polar}
\A\pc m=\A\plugin_{H_1}\Pl_{m_{H_1}}\plugin_{H_2}\Pl_{m_{H_2}}
\plugin_{H_3} \cdots\plugin_{H_n}\Pl_{m_{H_n}},
\end{equation}
if $\A=\set{H_1,\ldots,H_n}$ is the underlying simple arrangement.
\end{definition}

Note that $\rank(\A\pc m) = n_2(\A,m)+\rank \A$, where we let
\begin{equation}
\label{eq:nk}
n_k(\A,m)=\abs{\set{H\in\A\colon m_H\geq k}}
\end{equation}
Figure~\ref{fig:pc} shows the construction of the polarization $(\Pl_4,(3,2,1,3))$.
Furthermore, note that $\Pl_1$ is the unit of the parallel connection operad, 
so the terms given by hyperplanes $H$ with $m_H=1$ could be 
omitted from the formula \eqref{eq:polar} without change.

As observed in \S\ref{subsec:parallel}, the (projectivized) complement 
of a polarization splits as a product.  Write $P(\A)= \prod_{H\in\A} P_{m_H}$. 
By Proposition~\ref{prop:falk}, we have a homeomorphism
\begin{equation}
\label{eq:prod_decomp}
\xymatrix{
\theta \colon U(\A)\times P(\A) \ar^-{\cong}[r] &U(\A\pc m)}.
\end{equation}

By Proposition \ref{prop:H_1}, the isomorphism $\theta_*$
is given in degree $1$ by composing the relevant maps 
of type \eqref{eq:Phmap}:
\begin{equation}
\label{eq:theta_star}
\theta_*\colon H_1(U,\Z)\plugin_{H_1}H_1(P_{m_1},\Z)\plugin_{H_2}\cdots
\plugin_{H_n}H_1(P_{m_n},\Z)\to H_1(U(\A\pc m),\Z).
\end{equation}

The twisted Poincar\'e polynomial of the polarization is 
straightforward, using \eqref{eq:poin free}:
for any $\rho$ in the character torus of $U(\A\pc m)$,
\begin{equation}
\label{eq:poincare}
\poin(U(\A\pc m),\rho;x)=
\begin{cases}
\poin(U(\A),\rho;x)\cdot\prod_{\substack{H\in \A\colon\\ 
m_H\geq2}}(m_H-2)t & \text{for $\rho\neq\bo$;}\\
\poin(U(\A),\rho;x)\cdot\prod_{\substack{H\in \A\colon\\ 
m_H\geq2}}(1+(m_H-1)t) & \text{for $\rho=\bo$.}
\end{cases}
\end{equation} 

By applying formula \eqref{eq:pc of cv} repeatedly, 
we see that the characteristic varieties of a polarization
$\A\pc m$ can be obtained directly from the those of $\A$.

\subsection{The Milnor fibration of a polarization}
\label{subsec:mf polarized}

Our goal in this subsection is to relate the unreduced 
Milnor fiber $F(\A,m)$ associated with a multiarrangement 
$(\A,m)$ with the (reduced) Milnor fiber $F(\A\pc m)$ of the
simple arrangement $\A\pc m$ obtained by polarization.

Let $N=\sum_{H\in \A} m_H$. Using the homeomorphism 
$\theta$ of \eqref{eq:prod_decomp}, we have
a map
\begin{equation}
\label{eq:theta^star}
\xymatrix{
\theta^*\colon 
H^1(U(\A\pc m),\Z_N)\ar[r] & H^1(U(\A),\Z_N)\oplus H^1(P(\A),\Z_N),
}
\end{equation}
dual to \eqref{eq:theta_star}.  For each $H\in\A$, we define a 
cohomology class $\epsilon_H\in H^1(P_{m_H},\Z_N)$ in coordinates by 
\begin{equation}
\label{eq:eps}
\epsilon_H=(1-m_H,1,1,\ldots,1), 
\end{equation}

Recall from Proposition \ref{prop:mf} that the cover 
$F(\A,m)\to U(\A)$ is classified by the class 
$\delta_{\A,m}\in H^1(U(\A),\Z_N)$ given by 
$\delta_{\A,m}(x_H)=m_H \bmod N$ for $H\in\A$. 
Also recall that we write $\delta_{\A}=\delta_{\A,m}$ 
when all multiplicities are equal to $1$.

\begin{lemma}
\label{lem:img of delta}
We have $\theta^*(\delta_{\A\pc m})=(\delta_{\A,m},\epsilon)$.
\end{lemma}
\begin{proof}
We reduce our operadic isomorphism \eqref{eq:Phmap bis} from
Proposition~\ref{prop:proj operad2 dual} modulo $N$.  Since
$\delta_{\A\pc m}(x_K)=1$ for all $K\in\A\pc m$, the projection of 
$\theta^*(\delta_{\A\pc m})$ onto the summand $H^1(P_{m_H},\Z_N)$
agrees with $\epsilon_H$, using the first two cases of the 
formula \eqref{eq:pluginstar}, for each $H\in\A$.

Using the third case of formula \eqref{eq:pluginstar}, we have
\begin{align*}
\theta^*(\delta_{\A\pc m})(\overline{x}_H)&=\sum_{L\in\Pl_{m_H}}1\\
&=m_H,
\end{align*}
for each $H\in\A$, so the projection of $\theta^*(\delta_{\A\pc m})$
onto $H^1(U,\Z_N)$ equals $\delta_{\A,m}$.
\end{proof}

\begin{figure}
\newcommand{\drawBackbone}
{
\node[root] (h0) at (0,0) [label=left:-6] {};
\node[plain] (h1) at (0,1) [label=left:2] {};
\node[plain] (h2) at (0,2) [label=left:1] {};
\node[plain] (h3) at (0,3) [label=left:3] {};
\draw (h0) -- (h1) -- (h2) -- (h3);
}
\[
\begin{tikzpicture}[scale=0.75,baseline=(current bounding box.center),
plain/.style={circle,draw,inner sep=1.5pt,fill=white},
root/.style={circle,draw,inner sep=1.5pt,fill=black}]

\node[root] (h0) at (0,0) [label=left:-8] {};
\node[plain] (h1) at (0,1) [label=left:1] {};
\node[plain] (h2) at (0,2) [label=left:1] {};
\node[plain] (h3) at (0,3) [label=left:1] {};
\draw (h0) -- (h1) -- (h2) -- (h3);

\draw (h0) -- (1,0) node [plain,label=below:1] {} -- (2,0) 
                    node [plain,label=below:1] {} ;
\draw (h1) -- (1,1) node [plain,label=right:1] {};
\draw (h3) -- (1,3) node [plain,label=above:1] {} -- (2,3) 
                    node [plain,label=above:1] {} ;
\end{tikzpicture}
\quad\mapsto\quad
\begin{tikzpicture}[scale=0.75,baseline=(current bounding box.center),
plain/.style={circle,draw,inner sep=1.5pt,fill=white},
root/.style={circle,draw,inner sep=1.5pt,fill=black}]
\def\shift{1.2}
\drawBackbone
\node (k0) at (\shift,0) [root,label=below:-2] {};
\node (k1) at (\shift,1) [root,label=below:-1] {};
\node (k2) at (\shift,2) [root,label=right:0] {};
\node (k3) at (\shift,3) [root,label=above:-2] {};

\draw[dotted] (h0) -- (k0);
\draw[dotted] (h1) -- (k1);
\draw[dotted] (h2) -- (k2);
\draw[dotted] (h3) -- (k3);
\draw (k0) -- (\shift+1,0) node [plain,label=below:1] {} -- 
              (\shift+2,0) node [plain,label=below:1] {};
\draw (k1) -- (\shift+1,1) node [plain,label=below:1] {};
\draw (k3) -- (\shift+1,3) node [plain,label=above:1] {} -- 
              (\shift+2,3) node [plain,label=above:1] {};
\end{tikzpicture}
\]
\caption{An iterated application of $\plugin_x^*$}
\label{fig:pc2}
\end{figure}
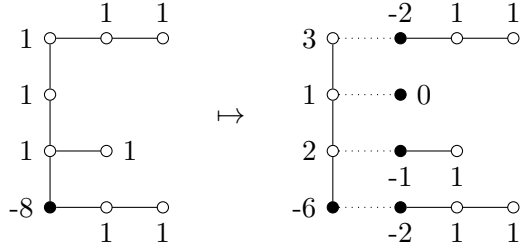

\begin{example}
Let $\A=\Pl_4$, and consider the polarization $(\A,m)$ for $m=(3,2,1,3)$
shown in Figure~\ref{fig:pc}.  Reducing the coefficients modulo $N=9$ in
Figure~\ref{fig:pc2} gives the coordinates of $\theta^*(\delta_{\A,m})$
shown for each factor.
\end{example}

\begin{lemma}
\label{lem:mf polar}
Let $(\A,m)$ be a multiarrangement, and let $\A\pc m$ be its polarization.  
We then have a pullback diagram between the respective Milnor covers, 
\[
\xymatrix{
F(\A,m) \ar[r] \ar[d] & F(\A\pc m) \ar[d] \\
U(\A) \ar[r]^(.45){j}  & U(\A\pc m), 
}
\]
where $j$ is the map given by choosing a base point 
in each space $P_{m_H}$, for $H\in \A$. 
\end{lemma}

\begin{proof}
Set $N=\sum_{H\in \A} m_H$ and $\B=\A\pc m$.  
The cover $F(\A,m) \to U(\A)$ is defined by the homomorphism 
$\delta_{\A,m} \colon H_1(U(\A),\Z)\surj \Z_N$, $x_H\mapsto m_H$, 
whereas the cover $F(\B) \to U(\B)$ is defined by the homomorphism 
$\delta_{\B} \colon H_1(U(\A'),\Z)\surj \Z_N$, described above.

The map $j^*$ is simply the projection of $\theta^*$ onto its first
factor.  Using Lemma~\ref{lem:img of delta}, we see that 
$j^*(\delta_{\B})=\delta_{\A,m}$, and this completes the proof. 
\end{proof}

\subsection{Torsion in the homology of the Milnor fiber of a polarization}
\label{subsec:mf polar}

Recall from Theorem~\ref{thm:milnor multi} that unreduced polynomials 
associated with certain multiarrangements give rise to torsion in the 
homology of their Milnor fibers.  We now use the polarization 
construction to replace multiarrangements with (larger) simple 
arrangements having the same property, thus completing the 
proof of Theorem \ref{thm:intro polar tors} from the Introduction.

\begin{theorem}
\label{th:torsmf}
Suppose $\A$ admits a pointed multinet structure, with multiplicity vector 
$m$, and distinguished hyperplane $H$.  Set $\A' =\A\setminus \{H\}$, 
and let $p$ be a prime dividing $m_H$.   Then, there is a choice of 
multiplicities $m'$ on $\A'$ such that the Milnor fiber of the polarization 
$\B=\A'\pc m'$ has $p$-torsion in homology, in degree $1+n_3(\A',m')$.
\end{theorem}

\begin{proof}
By Theorem~\ref{thm:milnor multi}, there exists a choice of multiplicities
$m'$ on $\A'$ such that $H_1(F(\A',m'),\Z)$ has $p$-torsion. 
Set $U=U(\A')$ and $\delta=\delta_{\A',m'}$. Then, the coefficient of $x$ in 
$\Delta^\k_{U,\delta}(\mathbf{u},x)-\Delta^\C_{U,\delta}(\mathbf{u},x)$ 
is strictly positive, provided $\ch(\k)=p$.

As noted in Remark~\ref{rem:avoid2}, we may assume additionally 
that each $m'_H\neq2$.  Form the polarization $\B=\A'\pc m'$, and let 
$P(\A')=\prod_{K\in\A'} P_K$.  By Lemmas~\ref{lem:img of delta} and  
\ref{lem:mf polar}, the Milnor fiber $F(\A',m')$ is the pullback of
$F(\B)$ along the map $j\colon U(\A') \to U(\B)$, and 
$(\theta^*)(\delta_{\B})=(\delta,\epsilon)$,
where $\epsilon$ is given by \eqref{eq:eps}.  

Let $N'=\sum_{K\in \A'} m'_K$ and $n_3=n_3(\A',m')$. 
By construction, $\epsilon_K$ lifts to a surjective homomorphism
$F_{m'_K-1}\surj \Z_{N'}$. Thus, by Proposition~\ref{prop:delta},
\begin{equation}
\label{eq:delta epsilon}
\Delta^\k_{P(\A'),\epsilon}(\mathbf{u},x)=u_1
\prod_{K\in\A'}(1+(m'_K-1)x)+cx^{n_3}, 
\end{equation}
where $c$ is a 
(strictly positive) multiple of $\sum_{1<k \mid N'}\phi(k)u_k$.
Then, by the product formula of
Proposition~\ref{prop:cover of product}, the coefficient 
of $x^{1+n_3}$ in
$\Delta^\k_{\B,\delta_{\B}}(\mathbf{u},x)-
\Delta^\C_{\B,\delta_{\B}}(\mathbf{u},x)$
is also strictly positive: that is, $H_{1+n_3}(F(\B),\Z)$
has $p$-torsion.
\end{proof}

\begin{example}
\label{ex:polar del B3}
We recall the deleted ${\rm B}_3$-arrangement $\A'$ from 
Examples~\ref{ex:deleted B3} and \ref{ex:mf deleted B3}.
By Theorem~\ref{th:torsmf}, there exists a polarization $\A'\pc m'$
for which $F(\A'\pc m')$ has $2$-torsion in homology.  As in 
Example~\ref{ex:mf deleted B3}, the choice 
$m'=(8,1,3,3,5,5,1,1)\in \Z^8$ avoids the multiplicity $2$.

Let $\B=\A'\pc m'$, an arrangement of $27$ 
hyperplanes in $\C^8$, with defining polynomial 
\begin{align}
\label{eq:big Q}
Q= &\: xy(x^2-y^2)(x^2-z^2)(y^2-z^2)
w_1w_2w_3w_4w_5 \cdot \\ \notag
&\: (x^2-w_1^2)(x^2-2w_1^2)(x^2-3w_1^2)(x-4w_1) 
((x-y)^2-w_2^2)((x+y)^2-w_3^2) \cdot \\ \notag
&\: ((x-z)^2-w_4^2)((x-z)^2-2w_4^2)((x+z)^2-w_5^2)((x+z)^2-2w_5^2).
\end{align}

Then $\poin(U(\B),x)=(1+7x+12x^2)(1+7x)(1+2x)^2(1+4x)^2$, 
and $n_3=5$.  Let $\k=\overline{\F_2}$. 
As we saw in Example~\ref{ex:mf deleted B3}, we have 
$\Delta^\k_{\A',\delta}(\mathbf{u},x)-
\Delta^\C_{\A',\delta}(\mathbf{u},x)=2u_3 x$. Hence, 
\begin{align*}
\Delta^\k_{\B,\delta_{\B}}(\mathbf{u},x)-
\Delta^\C_{\B,\delta_{\B}}(\mathbf{u},x) &=
\sum_{\rho\in\im(\widehat{\delta}_\C),\, \abs{\rho}=3}u_3 x\cdot
\poin^\C(P(\A'),\rho;x)\\
&= 2u_3\prod_{K\in\A'\colon m'_K \ge 3} (m'_K-2) x^{6},\\
&= 108 u_3x^6,
\end{align*}
by Proposition~\ref{prop:cover of product}.  That is, $H_6(F(\B),\Z)$ 
has $2$-torsion of rank $108$, which appears over $\k$ as an
extra factor of $(1+t+t^2)^{54}$ in the characteristic polynomial 
of the monodromy operator.  In fact,  $\Delta_6^{\k}(t)=
(t-1)^{11968} (t^2+t+1)^{54}$, using the Poincar\'{e} polynomial 
above to compute the multiplicity of the trivial representation.
\end{example}

\begin{figure}
\begin{tikzpicture}[scale=0.8]
\clip (0,0) circle (4.2);
\begin{scope}%[thick]
\draw (-1,-5) -- (-1,5);
\draw (0,-5) -- (0,5);
\draw (1,-5) -- (1,5);
\draw (-5,-1) -- (5,-1);
\draw (-5,0) -- (5,0);
\draw (-5,1) -- (5,1);
\draw (-5,-5) -- (5,5);
\draw (-5,5) -- (5,-5);
\end{scope}
% add some extra generic lines
\begin{scope}
\def\px{2}\def\py{2} % x-y
\foreach \theta in {105,130} {
  \draw[color=red,rotate around={\theta:(\px,\py)}] 
       (\px-6,\py) -- (\px+6,\py);
}
\def\px{-2.3}\def\py{2.3} % x+y
\foreach \theta in {50,68} {
  \draw[color=red,rotate around={\theta:(\px,\py)}] 
       (\px-6,\py) -- (\px+6,\py);
}
\def\px{0}\def\py{-2} % x
\foreach \j in {0,1,2,3,4,5,5,6} {
  \draw[color=blue,rotate around={4+9*\j:(\px,\py)}] 
       (\px-6,\py) -- (\px+6,\py);
}
\def\px{1}\def\py{2.6}  % x-z
\foreach \j in {0,1,2,3} {
  \draw[color=green!50!blue,rotate around={24+10*\j:(\px,\py)}] 
       (\px-8,\py) -- (\px+6,\py);
}
\def\px{-1}\def\py{-0.5} %x+z
\foreach \theta in {98,109,120,131} {
  \draw[color=green!50!blue,rotate around={\theta:(\px,\py)}] 
       (\px-6,\py) -- (\px+6,\py);
}
\end{scope}
\end{tikzpicture}
\caption{A planar section of $\A'\pc (8,1,3,3,5,5,1,1)$}
\label{fig:B3polar}
\end{figure}
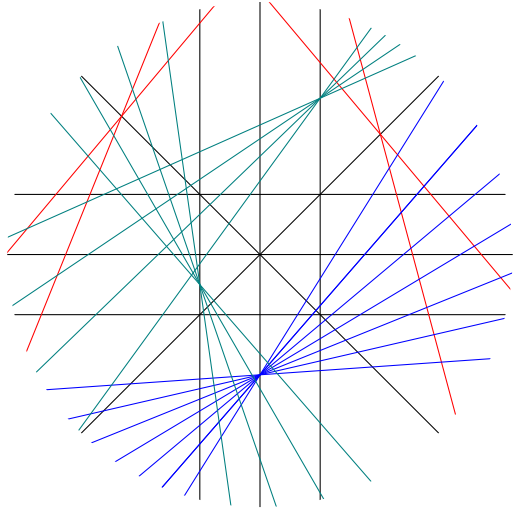

\begin{remark}
\label{rem:fiber-type}
We note that the deleted ${\rm B}_3$ arrangement $\A'$ 
is fiber-type, and thus the (projectivized) complement of 
$\B=\A'\pc m'$ is a $K(\pi,1)$ space.  By the Hamm-L\^{e} Theorem, 
the group $\pi$ is also the fundamental group of the arrangement 
of $27$ lines in $\P^2$ depicted in Figure~\ref{fig:B3polar}.  
Furthermore, the Milnor fiber of $\B$ is a $K(G,1)$ space, 
where $G$ denotes the kernel of  the ``diagonal'' homomorphism 
$\pi\to\Z_{27}$.  

We have shown that $H_6(G,\Z)$ has $2$-torsion.  In principle, 
this could be verified by finding a presentation for the group $G$, 
and then computing a free $\Z{G}$-resolution of $\Z$.  
However, we would be unable to predict the existence of such 
an example, or carry out the computations without the tools 
presented here.
\end{remark}

Since there are families of pointed multinets with arbitrary multiplicity on
the distinguished hyperplane (Example~\ref{ex:monomial}), the same 
polarization argument gives examples for all primes.

\begin{corollary}
\label{cor:pbig}
For every prime $p\ge 2$, there exists an arrangement $\B_p$ 
such that $H_{2p+3}(F(\B_p),\Z)$ has non-trivial $p$-torsion.
\end{corollary}

\begin{example}
\label{ex:torsmon}
Fix an odd prime $p$, and consider the deleted monomial arrangement 
$\A'_p$ from Example~\ref{ex:monomial}.  Let $q$ be a 
prime with $q-p>2$.  Then the multiplicity vector \eqref{eq:monomial m} 
avoids the multiplicity $2$, so we may polarize $\A'_p$ with the multiplicities 
\[
m'=(p,q-p,q-1,\ldots,q-1,1,\ldots,1,q,\ldots,q)
\]
given in \eqref{eq:monomial m}.  Here, $\B_p=\A_p\pc m'$ is an arrangement 
of $(2p+1)q$ hyperplanes, and $n_2=n_3=2p+2$, so the arrangement has 
rank $2p+5$.  Then $H_{2p+3}(F(\B_p),\Z)$ has nonzero $p$-torsion cycles.
Over $\overline{\F_p}$, these cycles are the span of 
order-$q$ monodromy eigenvectors.
\end{example}

This finishes the proof of Theorem \ref{thm:intro1} from the Introduction.

\subsection{Concluding remarks}
\label{subsec:conclude}
It should be noted that our method produces only high-dimensional 
examples of arrangements with torsion in the homology of 
the Milnor fiber.  The smallest example known to us is the 
arrangement $\B$ from Example~\ref{ex:polar del B3}, 
an arrangement of rank $8$ with non-trivial torsion in 
$H_6(F(\B),\Z)$. This leaves open the following questions.

\begin{question}
\label{q:H1}
Is there a hyperplane arrangement $\A$ whose Milnor fiber $F(\A)$ has
torsion in $H_1(F(\A),\Z)$?  
\end{question}

Given the relative importance of projective line arrangements, it is
also natural to ask the following.

\begin{question}
\label{q:r3}
Is there a rank-$3$ arrangement with torsion in the homology of its
Milnor fibre?  (Such an example would also resolve Question~\ref{q:H1}.)
\end{question}

Finally, our construction focusses entirely on ``non-modular torsion'':
that is, $p$-torsion in the homology of cyclic covers, where $p$ does 
not divide the order of the cover.  The modular case is qualitatively 
different.

\begin{question}
\label{q:mod}
Does there exist a hyperplane arrangement $\A$ for which 
$H_*(F(\A),\Z)$ has torsion of order $p$, and $p$ divides 
$\abs{\A}$?
\end{question}

\begin{ack}
We would like to thank the Max-Planck Institute for Mathematics
in Bonn for its support and hospitality while we carried out this
project.
\end{ack}

\newcommand{\arxiv}[1]
{\texttt{\href{http://arxiv.org/abs/#1}{arXiv:#1}}}
\newcommand{\doi}[1]
{\texttt{\href{http://dx.doi.org/#1}{doi:#1}}}

\renewcommand{\MR}[1]
{\href{http://www.ams.org/mathscinet-getitem?mr=#1}{MR#1}}

\newcommand{\MRh}[2]
{\href{http://www.ams.org/mathscinet-getitem?mr=#1}{MR#1 (#2)}}

\end{document}